\documentclass[11pt]{article}

\usepackage[margin=1in]{geometry}
\usepackage{amsmath,amssymb,amsthm,mathtools,mathrsfs,bm}
\usepackage{enumitem}
\usepackage{soul}
\usepackage{xcolor}
\usepackage{dsfont}
\usepackage[colorinlistoftodos]{todonotes} 
\definecolor{my-red}{rgb}{0.5,0.0,0.0}
\definecolor{my-blue}{rgb}{0.0,0.0,0.6}
\definecolor{my-green}{rgb}{0.0,0.5,0.0}
\definecolor{light-gray}{gray}{0.6}
\usepackage[colorlinks=true, citecolor=my-green, linkcolor=my-blue, urlcolor=magenta]{hyperref}

\numberwithin{equation}{section}

\newtheorem{theorem}{Theorem}[section]
\newtheorem{lemma}[theorem]{Lemma}
\newtheorem{proposition}[theorem]{Proposition}
\newtheorem{corollary}[theorem]{Corollary}
\theoremstyle{definition}

\newcommand{\eps}{\varepsilon}

\newcommand{\1}{\mathbf{1}}
\newcommand{\one}{\mathds{1}}

\newcommand{\N}{\mathbb{N}}
\newcommand{\Z}{\mathbb{Z}}
\newcommand{\Q}{\mathbb{Q}}
\newcommand{\R}{\mathbb{R}}

\newcommand{\E}{\mathbb{E}}
\renewcommand{\P}{\mathbb{P}}

 \newcommand{\cF}{\mathcal{F}}

 \newcommand{\cN}{\mathcal{N}}

\DeclareMathOperator{\sign}{sign}

\DeclareMathOperator{\op}{op}

\DeclareMathOperator{\TV}{TV}

\DeclareMathOperator{\diag}{diag}

\newcommand{\PP}{\mathbb{P}}

\newcommand{\EE}{\mathbb{E}}
\DeclareMathOperator{\Var}{Var}
\DeclareMathOperator{\Cov}{Cov}

\DeclareMathOperator{\Ent}{Ent}
\DeclareMathOperator{\ent}{ent}

\newcommand{\capOmega}{\widehat{\Omega}}
\newcommand{\gap}{\operatorname{gap}}
\newcommand{\iid}{\stackrel{iid}{\sim}}

\newcommand{\Semigroup}[1]{(P_t)_{t\ge0}}

\title{Rapid mixing within a phase for Ising models}

\author{Vasyl Paliy\thanks{Department of Mathematics, Brown University. Email:\url{vasyl_paliy@brown.edu}} \;\;\;\;\;\;\;\;\; Youngtak Sohn \thanks{Division of Applied Mathematics. Brown University, Email: \url{youngtak_sohn@brown.edu}}}
\date{\today}

\begin{document}
\maketitle
\vspace{-0.25cm}

\begin{abstract}
We study Glauber dynamics for Ising measures conditioned to have a sufficiently large majority of $+1$ spins. Under a sparse spectral width condition on the interaction matrix, we derive a Poincar\'{e} inequality for the conditioned measures from a magnetization drift condition, and a modified logarithmic Sobolev inequality from a stronger, entropy weighted condition. As applications, we obtain approximate posterior sampling in total variation distance for $\mathbb{Z}_2$ synchronization at high enough signal-to-noise ratio using $O(n\log n)$ Glauber updates from a spectral initialization. For ferromagnetic Ising models on random $d$-regular graphs, we prove rapid phase ordering down to inverse temperatures of order $1/d$, extending the temperature range of Gheissari, Sly, and Sohn (2025). Finally, we extend these inequalities to unconditioned Ising measures under an inward drift condition outside the conditioned region. This yields $O(n\log n)$ mixing of Glauber dynamics for the Sherrington–Kirkpatrick model at every fixed inverse temperature when the external field is sufficiently large.
\end{abstract}

\section{Introduction}
Glauber dynamics for Ising models has been studied in the last forty or more years, with a well-developed theory of rapid mixing at high temperatures under correlation decay conditions (see e.g. the survey~\cite{Martinelli-notes}). At low temperatures, however, phase coexistence can lead to exponentially slow mixing due to bottlenecks between phases. For example, ferromagnetic Ising models at low temperatures can exhibit two phases, with a majority of $+1$ and $-1$ spins, respectively, separated by nearly balanced configurations with exponentially small equilibrium probability (see, e.g., \cite[Chapter 15]{LP}). A natural question is whether transitions between these two phases are the only obstruction to rapid mixing. More precisely, if the dynamics is initialized with a bias toward one phase, does it rapidly approach the Gibbs measure conditioned on that phase, on a time scale much shorter than the time needed to leave it?

This question arises naturally in both statistical physics and Bayesian statistics. In statistical physics, the theory of phase ordering kinetics studies how order develops after cooling below the critical temperature, including how an initial bias affects the competition between phases~\cite{bray1994theory}.  In Bayesian statistics, the choice of initialization has long been recognized as important for posterior sampling~\cite{gelman1992inference}. For multimodal posteriors, a chain may explore the region around one mode while rarely visiting others~\cite{besag1993spatial}. This raises the question of whether a suitable initialization permits rapid (quasi-)equilibration within such a region, despite slow mixing globally.

Several rigorous results on Glauber dynamics for Ising models address this phase ordering question at positive temperatures. For the Curie–Weiss model, Levin, Luczak, and Peres~\cite{LLP} proved rapid mixing at low temperatures for dynamics restricted to positive magnetization. On the infinite regular tree, Caputo and Martinelli~\cite{CaMaTree} established rapid local convergence to the plus Gibbs measure from sufficiently biased product initializations at low temperatures. Gheissari and Sinclair~\cite{GhSi22} proved quasi-equilibration within the plus phase from the all-plus configuration on torus in $\Z^d$ and random regular graphs. On random regular graphs, Gheissari, Sly, and Sohn~\cite{GheissariSlySohn} established rapid phase ordering from arbitrary sufficiently biased initializations at low temperatures. More recently, Gheissari and Sly~\cite{GheissariSly} proved rapid phase ordering from sufficiently biased product initializations on $\Z^2$ throughout the low-temperature regime.

These results concern ferromagnetic models, and the proofs use various monotonicity arguments. In this paper, we study the phase ordering in Ising models with interactions of either sign. A motivating example from Bayesian statistics is the posterior distribution of $\mathbb Z_2$ synchronization~\cite{singer2011angular}. Let $\theta$ be the unknown vector uniformly distributed on $\{-1,+1\}^n$. We observe $Y\in \R^{n\times n}$ given by
\[
Y=\frac{\lambda}{n}\theta\theta^\top+W,
\]
where $\lambda>0$ is the signal-to-noise ratio and $W\sim \mathrm{GOE}(n)$ is independent noise; $W$ is a symmetric $n\times n$ matrix with independent entries above the diagonal with $(W_{ii})_{i\leq n}\stackrel{iid}{\sim}\cN(0,2/n)$ and $(W_{ij})_{i<j}\iid \cN(0,1/n)$. The posterior distribution $\pi_Y(\cdot)=\P(\theta=\cdot \mid Y)$ takes the form
\begin{equation}\label{eq:posterior:Z2}
\begin{split}
\pi_Y(\sigma)
&\propto\exp\bigg(\frac{\lambda}{2}\langle \sigma, Y\sigma\rangle\bigg)\\
&=\exp\bigg(\frac{\lambda^2}{2n}\big(\langle \sigma, \theta\rangle\big)^2+\frac{\lambda}{2}\langle \sigma, W\sigma\rangle \bigg), 
\end{split}
\end{equation}
where $\sigma\in \{-1,+1\}^n$. Thus, $\pi_Y$ is a Sherrington-Kirkpatrick (SK) model~\cite{sherrington_kirkpatrick75} with ferromagnetic bias aligned with $\theta$. The ``low temperature'' regime corresponds to $\lambda>1$, where the non-trivial estimation of $\theta$ (up to a global sign) becomes possible~\cite{DAM16}. Indeed, for every fixed $\lambda>1$, with high probability over $(\theta,W)$, the Glauber dynamics has worst-case mixing $\exp(\Omega(n))$ because of a bottleneck between configurations with positive overlap,
\[ \Omega_\theta^+:=\{\sigma:m_\theta(\sigma)\geq 0\}, \quad\textnormal{where}\quad m_\theta(\sigma):=\frac1n\langle\sigma,\theta\rangle, \]
and those in the opposite phase, $-\Omega_\theta^+$~\cite{lelarge2019fundamental}. 

For sufficiently large $\lambda$, Montanari and Wu~\cite{montanari2023posterior} obtained normalized $W_2$ sampling guarantees for $\pi_Y$ using a diffusion process and approximate message passing (covering more general priors). Our first main theorem shows that, for every fixed $\lambda\ge\lambda_0$, Glauber dynamics from any $X_0$ with $m_\theta(X_0)\ge C/\lambda$ quasi-equilibrates to $\pi_Y(\cdot\mid\Omega_\theta^+)$ in $O_\lambda(n\log n)$ updates. Spectral initialization~\cite{benaych2011eigenvalues} and an independent global sign flip then yield full posterior sampling in total variation distance in $O_{\lambda}(n\log n)$ updates, answering the MCMC sampling question mentioned in~\cite{Fan2021bayesian, montanari2023posterior}.

\begin{theorem}\label{theorem:z2-synchronization}
There are universal constants $C,\lambda_0>0$ such that for every
fixed $\lambda\ge\lambda_0$ the following holds with probability $1-o(1)$ over $(\theta,W)$ as $n\to\infty$. Let $(X_t)_{t\geq 0}$ be discrete time Glauber dynamics on $\pi_Y$ initialized from any $X_0$ with
$m_{\theta}(X_0)\geq C/\lambda$. Then, for some $C_0\equiv C_0(\lambda)>0$,
\begin{equation}\label{eq:z2-phase-ordering}
    \big\|
        \mathbb P_{X_0}(X_t\in \cdot\,)
        -\pi_Y(\,\cdot \mid \Omega_{\theta}^{+}\big)
    \big\|_{\mathrm{TV}}
    \le n^{-10}, \quad\textnormal{for all}\quad C_0n\log n\le t\le e^{n/C_0}.
\end{equation}
As a result, the following holds with probability $1-o(1)$. Let $v(Y)$ be a unit eigenvector corresponding to the largest eigenvalue of $Y$, and let $\sign(v(Y))$ denote its coordinatewise sign, with $\sign(0)=1$. Then,
\begin{equation}\label{eq:z2-spectral-sampling}
    \big\|
        \mathbb P_{\sign(v(Y))}
        (\xi X_t\in\cdot\,)-\pi_Y
    \big\|_{\mathrm{TV}}
    \le n^{-10} \quad\textnormal{for all}\quad C_0n\log n\le t\le e^{n/C_0},
\end{equation}
where $\xi\sim \textnormal{Unif}(\{-1,+1\})$ is independent of $(\theta, W)$ and $(X_t)_{t\geq 0}$.
\end{theorem}

\subsection{Core functional inequalities under sparse spectral width condition}
For the posterior distribution $\pi_Y$, conditional on $\theta$, the change of variables $\sigma_i\mapsto\theta_i\sigma_i$ maps $\theta$ to the all-ones vector $\mathbf1\in \R^n$. The overlap $m_\theta(\sigma)$ then becomes the magnetization $\frac{1}{n}\langle\mathbf1,\sigma\rangle$, so the two phases correspond to positive and negative magnetization.

We thus study general Ising measures conditioned to have sufficiently large magnetization. For a symmetric matrix $J\in\mathbb R^{n\times n}$ and $h\in\mathbb R^n$, define the measure $\mu_{J,h}$ on $\{-1,+1\}^n$ as
\[
    \mu_{J,h}(\sigma)\propto
    \exp(\langle J\sigma,\sigma\rangle+\langle h,\sigma\rangle),
    \qquad \sigma\in\Omega:=\{-1,+1\}^n.
\]
For $-n\le M\le n$ with $M\equiv n\pmod2$, consider the conditioned measure
\[
    \mu_{J,h,M}:=\mu_{J,h}(\,\cdot\mid\widehat\Omega_M), \quad\textnormal{where}\quad \widehat\Omega_M:=\left\{\sigma\in\Omega:\sum_{i=1}^n\sigma_i\ge M\right\}.
\]
Motivated by $\mathbb Z_2$ synchronization, we are particularly interested in the case where $\{\sigma:\sum_i \sigma_i\geq 0\}$ is one of the phases of $\mu_{J,h}$. We call $\widehat{\Omega}_M$ the \emph{core} of this phase. 

For a probability measure $\nu$ on $\Omega$, Glauber dynamics for $\nu$ chooses a coordinate uniformly and resamples its spin from the conditional law under $\nu$. Write $\sigma^{(i)}$ for the configuration obtained by flipping spin $i$. At a configuration $\sigma$ with $\nu(\sigma)>0$, the conditional flip probability is
\[
c_i(\sigma;\nu):=
\frac{\nu(\sigma^{(i)})}{\nu(\sigma)+\nu(\sigma^{(i)})}.
\]
The discrete time Glauber transition kernel is 
\[
P_\nu(\sigma,\tau):=
\begin{cases}
\frac{1}{n}c_i(\sigma;\nu)&\tau=\sigma^{(i)},\\
1-\frac1n\sum_{i=1}^n c_i(\sigma;\nu)&\tau=\sigma,
\end{cases}
\]
and $P_{\nu}(\sigma,\tau)=0$ otherwise. In particular, Glauber dynamics for $\mu_{J,h,M}$ is the reflected dynamics on $\widehat\Omega_M$: any update leaving  $\widehat{\Omega}_M$ is replaced by holding at the current configuration.

Our first goal is to prove rapid mixing of this reflected dynamics. We then give magnetization drift conditions under which the unrestricted dynamics on $\Omega$, started from sufficiently large magnetization, stays coupled to the reflected dynamics long enough to mix within $\widehat{\Omega}_M$. Similar arguments have appeared in the study of phase ordering~\cite{GhSi22, GheissariSlySohn}.

We obtain within-phase mixing bounds through functional inequalities on the core $\widehat\Omega_M$, namely Poincar\'{e} and modified logarithmic Sobolev inequalities (MLSI). Define the Dirichlet form
\begin{equation}\label{eq:dirichlet:form}
    \mathcal E_\nu(f,g):=
    \frac12\sum_{\sigma,\tau}\nu(\sigma)P_\nu(\sigma,\tau)
    \bigl(f(\sigma)-f(\tau)\bigr)\bigl(g(\sigma)-g(\tau)\bigr),
    \qquad \mathcal E_\nu(f):=\mathcal E_\nu(f,f).
\end{equation}
The spectral gap and
MLSI constants are defined by
\[
    \gap(\nu):=\inf_{f:\Omega\to \R}
        \frac{\mathcal E_\nu(f)}{\Var_\nu(f)},
    \qquad
    \rho_{\mathrm{LS}}(\nu):=
    \inf_{f:\Omega\to (0,\infty)}
        \frac{\mathcal E_\nu(f,\log f)}{\Ent_\nu(f)},
\]
where $\Ent_\nu(f):=\nu[f\log f]
    -\nu[f]\log\nu[f]$. These quantities give standard bounds on the worst-case total variation mixing time. Define
    \[
    t_{\mathrm{mix}}(P_\nu,\eta):=
    \min\left\{t\geq 0:
    \max_{\sigma\in \Omega:\,\nu(\sigma)>0}
    \left\|P_\nu^t(\sigma,\cdot)-\nu\right\|_{\mathrm{TV}}\le\eta
    \right\}, 
    \]
    and let $\nu_{\min}:=\min_{\sigma:\nu(\sigma)>0}\nu(\sigma)$. For $0<\eta\leq 1/2$, we have
    \begin{equation}\label{eq:mixing-time-bounds}
    t_{\mathrm{mix}}(P_\nu,\eta)
    \lesssim\min\left\{
    \frac{\log(1/\nu_{\min})+\log(1/\eta)}{\gap(\nu)}\,,\,
    \frac{\log\log(1/\nu_{\min})+\log(1/\eta)}
    {\rho_{\mathrm{LS}}(\nu)}
    \right\}.
    \end{equation}
    See e.g. \cite[Chapter~12]{LP} and \cite{bobkov2006modified}.

    For unconditioned measure $\mu_{J,h}$, recent advances~\cite{BB19, EKZ, ChenEldan} established that MLSI constant satisfies $\rho_{{\mathrm LS}}=\Omega(n^{-1})$ under the \emph{full spectral width condition} $\lambda_{\max}(J)-\lambda_{\min}(J)<1/2$. The challenge is that this condition fails (with high probability) for $J=\frac{\lambda^2}{2n}\theta\theta^\top+\frac{\lambda}{2}W$, the interaction matrix of the posterior distribution $\pi_Y$ in \eqref{eq:posterior:Z2}, when $\lambda$ is large (the low temperature regime). 
    
    For the core-conditioned measure  $\mu_{J,h,M}$, we instead consider a spectral condition involving only principal submatrices of size at most $n-M$. For $0\le M<n$, define the \emph{sparse spectral width}
\[
    \omega_M(J):=
    \max_{\substack{\varnothing\ne T\subseteq[n]\\|T|\le n-M}}
    \bigl(\lambda_{\max}(J_{T,T})-\lambda_{\min}(J_{T,T})\bigr),
\]
where $J_{T,T}$ denotes $T\times T$ submatrix of $J$.
Here, the size $n-M$ comes from the geometry of $\widehat{\Omega}_M$: $R_M:=\frac{n-M}{2}$ is the largest number of minus spins allowed in $\widehat\Omega_M$, so any two configurations in $\widehat\Omega_M$ differ on at most
$2R_M=n-M$ coordinates. Indeed, we prove in Proposition~\ref{prop:sparse-covariance} that 
\[
8\omega_M(J)\leq 1-\eps
\]
implies spectral independence~\cite{anari2024spectral} of $\mu_{J,h,M}$. In particular, $\|\Cov(\mu_{J,h,M})\|_{\op}$ is uniformly bounded over $h\in \R^n$.

Replacing the full spectral width condition with a condition on $\omega_M(J)$ allows our results to apply to $\pi_Y$ at large fixed $\lambda$: we choose $M$ large enough that $8\omega_M(J)\le1-\eps$, but small enough that $\widehat\Omega_M$ still contains the positive posterior mode. To state the next result, define
\[
    \phi(u):=\frac1{1+e^{2u}},
    \qquad
    \psi(u):=\phi(u)\log\frac1{\phi(u)(1-\phi(u))}.
\]
The following theorem establishes functional inequalities on the core $\widehat\Omega_M$ and lies at the heart of our proof of Theorem~\ref{theorem:z2-synchronization}.

\begin{theorem}\label{theorem:core-func-ineq}
Let $\eps\in(0,1)$ and $M\in\{0,\ldots,n-2\}$ satisfy
$M\equiv n\pmod2$. Suppose $J\in \R^{n\times n}$ is symmetric and satisfies $8\omega_M(J)\le1-\eps$. Then, for some universal constants $c,C>0$, the following hold with $\delta_*:=1/(4e^2)$ and $R_M:=\frac{n-M}{2}$.
\begin{itemize}
    \item[(a)] For every $h\in \R^n$ such that the magnetization drift condition
    \begin{equation}\tag{H1}\label{assumption:stability}
        \sum_{i=1}^n\phi\Big(h_i+2\sum_{j\neq i}J_{ij}\Big)\le\delta_*R_M
    \end{equation}
    holds, we have
    \[
        \gap(\mu_{J,h,M})\ge\frac cn
        \exp\left(-\frac{C\|J\|_{\mathrm{op}}}{\eps}\right).
    \]

    \item[(b)] For every $h\in \R^n$ such that the entropy-weighted drift condition
    \begin{equation}\tag{H1-E}\label{assumption:stability-entropy}
        \sum_{i=1}^n\psi\Big(h_i+2\sum_{j\neq i}J_{ij}\Big)\le\delta_*R_M
    \end{equation}
   holds, we have
    \[
        \rho_{\mathrm{LS}}(\mu_{J,h,M})\ge\frac cn
        \exp\left(-\frac{C\|J\|_{\mathrm{op}}}{\eps}\right).
    \]
\end{itemize}
\end{theorem}
The proof of Theorem~\ref{theorem:core-func-ineq} is based on the stochastic localization scheme by Chen and Eldan~\cite{ChenEldan}. We refer to Section~\ref{sec:proof-technique} for a proof overview.

We remark that the conditions~\eqref{assumption:stability} and \eqref{assumption:stability-entropy} may be viewed as drift conditions as follows. For $\sigma\in \Omega$, define the local field
\[
\ell_i(\sigma):=h_i+2\sum_{j\ne i}J_{ij}\sigma_j.
\]
With this notation, the probability of resampling spin
$i$ to $-1$ in a Glauber update for $\mu_{J,h}$ is
\[
\mu_{J,h}(\sigma_i=-1\mid \sigma_{-i})=\phi\big(\ell_i(\sigma)\big).
\]
Thus, the probabilities of increasing and decreasing the magnetization
\[
m(\sigma):=\frac{1}{n}\sum_i \sigma_i
\]
in one update are
\[
q_+(\sigma):=\frac1n\sum_{i:\sigma_i=-1}
    \bigl(1-\phi(\ell_i(\sigma))\bigr),
\qquad
q_-(\sigma):=\frac1n\sum_{i:\sigma_i=+1}
    \phi(\ell_i(\sigma)).
\]
In particular, the magnetization drift (scaled by $n$) is $n(P_{\mu_{J,h}}m(\sigma)-m(\sigma))=2(q_{+}(\sigma)-q_{-}(\sigma))$. Condition~\eqref{assumption:stability} states that
$q_-(\mathbf1)\le\delta_*R_M/n$, and since $q_{+}(\1)=0$, this bounds the downward drift at $\mathbf1$ by $2\delta_* R_M/n$. Using \eqref{assumption:stability} and $\omega_M(J)<1/2$, we show in Lemma~\ref{prop:boundary-stability} below that $q_-(\sigma)\le\frac45q_+(\sigma)$ for $\sigma\in \widehat{\Omega}_M$ with more than $2R_M/3$ minus spins. Therefore, under these assumptions, the dynamics has a uniform bias toward larger magnetization
near the boundary of $\widehat\Omega_M$.

Condition~\eqref{assumption:stability-entropy} weights each
minus-update probability $\phi(\ell_i(\1)$ by
$\log(1/[\phi(\ell_i(\1))(1-\phi(\ell_i(\1)))])$. This factor reflects the extra cost of entropy estimates, compared with variance estimates, for strongly biased updates (cf. Lemma~\ref{lemma:entropic-boundary-reduction}). Since $\psi(u) \ge \log(4)\phi(u)$, \eqref{assumption:stability-entropy} implies \eqref{assumption:stability}.

Finally, the drift conditions such as \eqref{assumption:stability} and \eqref{assumption:stability-entropy} are necessary in the following sense: even when $J=0$, taking $h_i=-H$ for every $i$ and $R_M\ge1$, it is not difficult to see from Cheeger's inequality that
\begin{equation}\label{eq:failure:uniform:gap}
\gap(\mu_{0,h,M})
\le \frac{2R_M}{n(1+e^{2H})}.
\end{equation}
Since $\rho_{\mathrm{LS}}(\cdot)\leq 2\gap(\cdot)$~\cite{bobkov2006modified}, both constants deteriorate exponentially as $H\to\infty$. When the external field is strongly negative, the measure concentrates near the boundary, while moving between boundary configurations requires rare flips against the field. Thus, additional assumptions on $h\in \R^n$ are needed to guarantee a gap of order $1/n$, and these drift conditions exclude this obstruction.

\subsection{Lifting the core functional inequalities}

To prove Theorem~\ref{theorem:z2-synchronization}, we apply Theorem~\ref{theorem:core-func-ineq} (after the change of variables $\sigma_i\mapsto \theta_i\sigma_i$) to a core $\widehat{\Omega}_M$ with $M/n=1-\Theta(1/\lambda^3)$. To achieve phase ordering from initial magnetization as small as $\Theta(1/\lambda)$, we then \emph{lift} these functional inequalities in Theorem~\ref{theorem:core-func-ineq} to a larger region under additional drift conditions. In doing so, we obtain MLSI for the SK model that may be of independent interest.

Let $P_{J,h}\equiv P_{\mu_{J,h}}$ denote the discrete time Glauber kernel of $\mu_{J,h}$. To extend the functional inequalities from the core $\widehat\Omega_M$ to a larger region $\widehat\Omega_S$, where $-n\le S<M$, we assume that the reflected dynamics has inward drift toward the core: for some $\zeta\in(0,1)$,
\begin{equation}\tag{H2}\label{assumption:tail-stability}
q_{+}(\sigma)
-q_{-}(\sigma)\one_{m(\sigma)>S/n}
\ge\frac{\zeta}{2}\big(1-m(\sigma)\big)
\qquad \sigma\in\widehat\Omega_S\setminus\widehat\Omega_M
\end{equation}
The indicator accounts for reflection at $m(\sigma)=S/n$. Thus, this condition requires upward magnetization drift throughout $\widehat\Omega_S\setminus\widehat\Omega_M$. For MLSI, we also assume that, for some $\eta>0$,
\begin{equation}\tag{H3}\label{assumption:H4}
    \max_{\sigma\in\widehat\Omega_S\setminus\widehat\Omega_M}
    \log\bigg(\frac1{\mu_{J,h,S}(\sigma)}\bigg)\le \eta n.
\end{equation}
By combining Theorem~\ref{theorem:core-func-ineq} with Lyapunov drift arguments~\cite{MT09, TM22}, we prove the following result in Section~\ref{sec:lifting}.
\begin{theorem}\label{theorem:lifting}
Let $J\in \R^{n\times n}$ be symmetric, $h\in\mathbb R^n$, and $\eps,\zeta\in(0,1)$.
Let $M,S$ be integers of the same parity as $n$, with $0\le M\le n-2$ and $-n\le S<M$. Suppose that $8\omega_M(J)\le1-\eps$ and
\eqref{assumption:tail-stability} holds for $\zeta$ from $\widehat\Omega_M$ to
$\widehat\Omega_S$. For some universal constants $c,C>0$, we have the following.
\begin{itemize}
\item[(a)] If \eqref{assumption:stability} holds, then
\[
    \gap(\mu_{J,h,S})\ge\frac{c\zeta}{n}
        \exp\left(-\frac{C\|J\|_{\mathrm{op}}}{\eps}\right).
\]
\item[(b)] If \eqref{assumption:stability-entropy} holds,
$M\leq (1-\vartheta) n$ for some $\vartheta\in(0,1)$, and
\eqref{assumption:H4} holds with $\eta>0$, then
\[
    \rho_{\mathrm{LS}}(\mu_{J,h,S})\ge
    \frac{c_{\zeta,\vartheta,\eta}}n
    \exp\left(-\frac{C\|J\|_{\mathrm{op}}}{\eps}\right),
\]
where $c_{\zeta,\vartheta,\eta}>0$ depends only on $\zeta, \vartheta,\eta>0$.
\end{itemize}
\end{theorem}
Taking $S=-n$, so that $\mu_{J,h,S}=\mu_{J,h}$, and $M=n-2\lfloor\delta n/2\rfloor$ in Theorem~\ref{theorem:lifting} yields the following corollary upon verifying \eqref{assumption:stability} and \eqref{assumption:tail-stability}, and, under the additional assumptions, \eqref{assumption:stability-entropy} and \eqref{assumption:H4}.

\begin{corollary}\label{cor:high-field-func-ineq}
Let $J\in \R^{n\times n}$ be a symmetric matrix, and denote $\gamma:=\omega_0(J)\equiv \lambda_{\max}(J)-\lambda_{\min}(J)$. Let $h\in \R^n$ and fix $\eps,\delta\in(0,1)$. Assume that for $n\geq 2/\delta$
\begin{equation}\label{assumption:high-field-corollary}
\begin{gathered}
    8\omega_{n-\lfloor\delta n\rfloor}(J)\le1-\eps\qquad\textnormal{and}\qquad
    \sum_{i=1}^n\phi\Bigl(h_i+2\sum_{j\ne i}J_{ij}-8\gamma\Bigr)
    \le c\delta n,
\end{gathered}
\end{equation}
where $c>0$ is a universal constant. Then,  $\gap(\mu_{J,h})\geq c_{\eps,\gamma}/n$ for some $c_{\eps,\gamma}>0$ that depends only on $\eps,\gamma>0$. If in addition, for some $H > 0$,
\[
    \frac{1}{n} \sum_{i=1}^n h_i \leq H \quad \textnormal{and} \quad \sum_{i=1}^n\psi\Bigl(h_i+2\sum_{j\ne i}J_{ij}\Bigr) \leq c\delta n,
\]
then $\rho_{\mathrm{LS}}(\mu_{J,h})\geq c_{\eps,\gamma,\delta,H}/n$ for some $c_{\eps,\gamma, \delta,H}>0$ that depends only on $\eps,\gamma,\delta, H>0$.
\end{corollary}

We next apply this result to the Sherrington-Kirkpatrick model~\cite{sherrington_kirkpatrick75} with inverse temperature $\beta>0$, defined by $J=\beta W/2$ and $W\sim \mathrm{GOE}(n)$. For every $\beta<1/4$ and $h=0$, it was shown in \cite{anari2021entropic, ChenEldan} that $\rho_{{\mathrm LS}}(\mu_{\beta W/2, 0})\gtrsim_{\beta} 1/n$ with high probability over $W$. The range $\beta<1/4$ stems from the full spectral width condition $\lambda_{\max}(J)-\lambda_{\min}(J)<1/2$, and it is widely believed that the rapid mixing of Glauber dynamics should hold throughout $\beta<1$~\cite{anari2021entropic}.  Here, we obtain MLSI constant of order $1/n$ at every fixed inverse temperature when the external field is sufficiently large at every site and the average field is $O(1)$.

\begin{corollary}\label{corollary:SK-high-field-mlsi}
There is a universal constant $C>0$ such that, for every fixed $\beta\ge0$ and $H>0$, there exists $c_0=c_0(\beta,H)>0$ for which the following holds with probability $1-o(1)$ over $W$ as $n\to\infty$. Uniformly over $h\in\R^n$ satisfying
\[
    \min_{1\le i\le n}h_i\ge C\bigl(1+\beta\sqrt{\log(1+\beta)}\bigr)
    \qquad\textnormal{and}\qquad \frac1n\sum_{i=1}^n h_i\le H,
\]
 we have 
 \[
 \rho_{\mathrm{LS}}(\mu_{\beta W/2,h})\ge\frac{c_0}{n}.
 \]
 In particular, the discrete time Glauber dynamics mixes in $O_{\beta,H}(n\log n)$ time.
\end{corollary}

    In concurrent work,  Bandeira, El Alaoui, and R{\"o}dder \cite{BER26} establish polynomial-time mixing of Glauber dynamics for the (unconditioned) Ising model $\mu_{J,h}$ when  
    \[
    \max_{T\subseteq [n]:|T|\leq \delta n}\|J_{T,T}\|_{\op}<\frac{1}{10}
    \]
    in our normalization and $h_i= H$ for all $i$, where $H$ is sufficiently large depending on $\delta$ and $\|J\|_{\op}$. For zero-diagonal $J$, the operator norm and spectral width of each principal submatrix differ by at most a factor of two. Thus, the above condition is equivalent, up to universal constants, to our sparse spectral width condition $8\omega_M(J)<1-\eps$. In the terminology of~\cite{ChenEldan}, their proof uses coordinate-by-coordinate localization on $\Omega$, whereas ours uses linear-tilt stochastic localization on $\widehat\Omega_M$; see Section~\ref{sec:proof-technique} for a proof overview. Theorem~\ref{theorem:z2-synchronization} answers the open question in \cite{BER26} about Glauber mixing for $\mathbb Z_2$ synchronization at large $\lambda$, through quasi-equilibration from sufficiently biased initializations.

\subsection{Ferromagnetic Ising model on random regular graphs}
The ferromagnetic Ising model on random regular graph is another model that exhibits phase ordering at low temperature. Let $\mathcal G_d(n)$ be the uniform distribution on simple $d$-regular
graphs on $n$ vertices with $nd$ even. For $G\sim \mathcal G_d(n)$, consider the ferromagnetic Ising
measure
\begin{equation}\label{eq:rrg-ising}
 \pi_{G,\beta}(\sigma)
 \propto\exp\Big(\beta\sum_{i\sim j}\sigma_i\sigma_j\Big),
     \qquad \sigma\in\Omega\equiv \{-1,+1\}^n.
\end{equation}
The high-temperature regime for this model is $(d-1)\tanh\beta<1$, where Mossel and Sly~\cite{MS} established rapid mixing of Glauber dynamics. Beyond this regime, Bauerschmidt, Bodineau, and Dagallier~\cite{BaBoDa-Kawasaki} proved rapid mixing of Kawasaki dynamics for $\beta<1/(8\sqrt{d-1})$. More recently, Gheissari, Sly, and Sohn~\cite{GheissariSlySohn} established rapid phase ordering for Glauber dynamics from sufficiently biased initializations for all $\beta\gtrsim \log d/d$, using a temporal recursion to control minus-spin clusters in spacetime.

As discussed in~\cite{GheissariSlySohn}, the extra factor $\log d$, compared with the critical inverse temperature $\operatorname{arctanh}
(1/(d-1))=(1+o_d(1))/d$, arises from their proof method of controlling the of minus-spin spacetime clusters through a dominating dynamics. Our next result removes this $\log d$ factor, establishing rapid phase ordering down to $\beta\asymp 1/d$. Let $m(\sigma):=n^{-1}\sum_{i=1}^n\sigma_i$ denote the magnetization, and write
\[ 
    \pi_{G,\beta}^+:=\pi_{G,\beta}(\,\cdot\mid\Omega^+),\quad\textnormal{where}\quad \Omega^+:=\{\sigma\in\Omega:m(\sigma)\ge0\}.
\]
Let $(X_t)_{t\in\mathbb Z_{\geq0}}$ be the discrete-time Glauber dynamics for $\pi_{G,\beta}$.

\begin{theorem}\label{theorem:rrg-ordering}
There are universal constants $d_0,c_0>0$ such that the following holds. Suppose that $d\ge d_0$ and $\frac{c_0}{d}\le\beta\le\frac1{16\sqrt d}$.
Let $M$ be the smallest integer with $M\equiv n\pmod2$ and
$M\ge3\beta n/2$. There exists $C=C(\beta,d)>0$ such that,
if $G\sim\mathcal G_d(n)$, with probability $1-o(1)$ as
$n\to\infty$, the discrete-time Glauber dynamics on $\widehat\Omega_M$ satisfies
\begin{equation}\label{eq:rrg-basin-mlsi}
    \rho_{\mathrm{LS}}\bigl(
        \pi_{G,\beta}(\,\cdot\mid\widehat\Omega_M)
    \bigr)\ge\frac{1}{Cn}.
\end{equation}
Moreover, with probability $1-o(1)$, we have for every $X_0$ with $m(X_0)\ge2\beta$,
\begin{equation}\label{eq:rrg-phase-ordering}
    \bigl\|\PP_{X_0}(X_t\in\cdot)-\pi_{G,\beta}^+\bigr\|_{\TV}\le n^{-10},\qquad Cn\log n\le t\le e^{n/C},\quad t\in\mathbb Z_{\ge0}.
\end{equation}
\end{theorem}

Note that Theorem~\ref{theorem:rrg-ordering} permits an initial bias of order $\beta$ throughout $c_0/d\le\beta\le1/(16\sqrt d)$. In~\cite[Theorem~2]{GheissariSlySohn}, the sufficient bias is of order $\max\{d^{-1/2},(\beta d)^{-1}\}$; our bound improves upon this bias when $\beta\le cd^{-1/2}$ for sufficiently small $c>0$. The improvement follows from the refined magnetization drift estimate in Section~\ref{sec:applications}, which verifies the inward drift condition~\eqref{assumption:tail-stability}. It would be interesting to determine whether these bias scales are optimal, by constructing initializations with bias $C\beta$, or $C/\lambda$ in Theorem~\ref{theorem:z2-synchronization}, that fail to quasi-equilibrate rapidly to the favored phase.

\section{Proof overview}\label{sec:proof-technique}
This section proves Theorem~\ref{theorem:core-func-ineq} with intermediate steps deferred to Sections~\ref{sec:spectral-independence}-\ref{sec:boundary}. We then deduce Theorem~\ref{theorem:z2-synchronization} from Theorem~\ref{theorem:core-func-ineq} and Theorem~\ref{theorem:lifting}, deferring the proof of the required estimate, Proposition~\ref{prop:z2-inputs}, to Section~\ref{sec:applications}. The proof of Theorem~\ref{theorem:lifting} is in Section~\ref{sec:lifting}, while the proofs of Corollaries~\ref{cor:high-field-func-ineq}, \ref{corollary:SK-high-field-mlsi}, and Theorem~\ref{theorem:rrg-ordering} are in Section~\ref{sec:applications}.
\subsection{Proof of the core functional inequalities}

The starting point of our proof of Theorem~\ref{theorem:core-func-ineq} is the following result, which is  a large-magnetization conditioned analog of \cite[Lemma 5.2]{ChenEldan}. Its proof in Section~\ref{sec:spectral-independence} combines
the decomposition obtained in~\cite{BB19} with the method of ``duplicate''
variables~\cite{percus1975correlation, lebowitz1974ghs}, where the sparse spectral width condition $8\omega_M(J)\leq1-\eps$
controls the interaction on the disagreement set of two i.i.d. samples;
see~\eqref{eq:conditional:law:XB:given:XA}. Recall $\mu_{J,h,M}(\cdot)=\mu_{J,h}(\cdot\mid \widehat{\Omega}_M)$ and $\widehat\Omega_M=\{\sigma: \sum_{i}\sigma_i\geq M\}$.

\begin{proposition}\label{prop:sparse-covariance}
Assume that $8\omega_M(J)\le1-\eps$ for some $0\leq M\leq n-2$ and $\eps\in(0,1)$. Then, for every $h\in\R^n$,
\[
    \|\Cov(\mu_{J,h,M})\|_{\mathrm{op}} \leq \frac{2}{\eps} 
\]
\end{proposition}

   We then apply the stochastic localization scheme by Eldan~\cite{eldan2013thin}, in which this covariance norm bound provides the \emph{approximate conservation of variance and entropy} (cf. Proposition~\ref{fact:EC-conservation}). For the sake of the exposition, let $\mu$ be a probability measure supported on $\Omega\equiv \{-1,+1\}^n$, and define the measure-valued process $(\nu_t)_{t\geq 0}$ by the following system of SDE:
        \begin{equation}\label{eq:def:SL}
        d\nu_t(x)=\nu_t(x)\big\langle x- b(\nu_t), C_t dB_t\big\rangle\,,\qquad\forall x\in \Omega\,,\qquad \nu_0=\mu.
        \end{equation}
     Here, $(B_t)_{t\geq 0}$ is a standard Brownian motion on $\R^n$, $b(\nu_t)=\sum_{x\in \Omega} x \nu_t(x)$ is the barycenter, and $(C_t)_{t\geq 0}$ is an adapted process of $n\times n$ positive semi-definite matrices. The uniqueness and existence of $(\nu_t)_{t}$ was established in \cite{eldan2013thin, eldan20clt}. By Ito's formula (see \cite[Fact~2.14]{ChenEldan}), $\nu_t(\cdot)$ is a probability measure on $\Omega$, which tilts $\mu$ by a random quadratic factor
        \begin{equation} \label{eq:SL}
          \nu_t(x)\propto e^{-\frac{1}{2}\langle x, Q_tx\rangle + \langle h_t,x\rangle}\mu(x)\,,\quad Q_t=\int_{0}^{t}C_s^2 ds\,,\quad h_t=\int_{0}^{t} \Big(C_s dB_s+C_s^2 b(\nu_s) ds\Big).
        \end{equation}
       For $\mu=\mu_{J,h}$ with $J\succ0$, choosing $C_t\equiv(2J)^{1/2}$ gives $\nu_t=\mu_{(1-t)J,h+h_t}$. At time $t=1$, the quadratic interaction is removed, so $\nu_1$ is a product measure. It therefore satisfies tensorization of variance and entropy~\cite{CMT15}, and hence Poincar\'{e} and MLSI \emph{almost surely} (i.e. pathwise) in the external field $h+h_1$. This interpolation has been central to recent advances on the mixing of Glauber dynamics for Ising models (see, e.g.,~\cite{AKV24, caputo2025factorizations,chen2025rapidthreshold}).

     To establish Poincar\'{e} and MLSI for $\mu_{J,h,M}$, a natural approach is to apply stochastic localization on $\widehat\Omega_M$, interpolating $\mu_{J,h,M}$ with a product measure conditioned on $\widehat\Omega_M$. However, as~\eqref{eq:failure:uniform:gap} shows, these conditioned product measures \emph{do not} satisfy tensorization of variance uniformly in the external fields. Thus, the functional inequalities at $t=1$ needed for the pathwise localization argument~\cite{Eldan2023analysis,ChenEldan} need not hold almost surely with uniform constants. Even a high-probability bound would leave a difficulty as these arguments require control of variance and entropy in expectation over the localization process, uniformly over $f:\widehat\Omega_M\to \R$, whereas exceptional paths may contribute substantially to these expectations for some $f$.

   To address this difficulty, we consider the localization process $\nu_t$ on the ambient space $\Omega$ together with its restriction $\widehat\nu_t:=\nu_t(\cdot\mid\widehat\Omega_M)$. As shown in Section~\ref{sec:localization}, under a change of probability law, $\widehat\nu_t$ satisfies the localization SDE with its own barycenter, starting from $\mu_{J,h,M}$. We then extend functions constantly outside $\widehat\Omega_M$ and apply tensorization of variance and entropy at $t=1$, obtaining inequalities with \textit{boundary terms}, which are defined below.

    To make this precise, we consider a general setting where $\nu$ is a given probability measure on $\Omega\equiv \{-1,+1\}^n$ and a subset $\widehat\Omega\subset \Omega$ is viewed as the core, with $\nu(\widehat\Omega)>0$. We are interested in obtaining functional inequalities for the conditioned measure $\widehat{\nu}:=\nu(\,\cdot\mid \widehat\Omega)$. For a function $f:\widehat\Omega\to \R$ and $a\in \R$, define the constant extension
    \begin{equation}\label{eq:constant:extension}
    f_a(\sigma)=
    \begin{cases}
        f(\sigma) & \sigma\in \widehat\Omega\\
        a & \sigma\notin \widehat\Omega.
    \end{cases}
    \end{equation}
Since $P_\nu$ and $P_{\widehat\nu}$ agree on flips within $\widehat\Omega$, reversibility and the definition of the Dirichlet form~\eqref{eq:dirichlet:form} give
    \[
    \begin{split}
    \mathcal{E}_{\nu}(f_a)
    &=\frac{1}{2}\sum_{\sigma, \tau \in \widehat\Omega}\nu(\sigma)P_{\nu}(\sigma, \tau)\big(f(\sigma)-f(\tau)\big)^2+\sum_{\sigma\in \widehat\Omega, \tau \notin \widehat\Omega}\nu(\sigma)P_{\nu}(\sigma, \tau)\big(f(\sigma)-a\big)^2\\
    &=\nu(\widehat\Omega)\bigg(\mathcal{E}_{\widehat{\nu}}(f)+\sum_{\sigma\in \widehat\Omega, \tau \notin \widehat\Omega}\widehat\nu(\sigma)P_{\nu}(\sigma, \tau)\big(f(\sigma)-a\big)^2\bigg), 
    \end{split}
    \]
    This motivates the \emph{quadratic boundary term} $\mathcal{Q}_{\nu}^{\widehat\Omega}(f)$ defined by
    \begin{equation}\label{eq:arbitrary-quadratic-boundary}
         \mathcal Q_{\nu}^{\widehat\Omega}(f):=\inf_{a\in\R}\mathcal Q_{\nu}^{\widehat\Omega}(f;a),\quad\textnormal{where}\quad\mathcal Q_{\nu}^{\widehat\Omega}(f;a):=\sum_{\sigma\in \widehat\Omega,\tau\notin \widehat\Omega}\widehat\nu(\sigma)P_{\nu}(\sigma, \tau)\big(f(\sigma)-a\big)^2.
    \end{equation}
    It is worth noting that, in general, $\mathcal{Q}_{\nu}^{\widehat\Omega}(f;a)$ is not determined by $\widehat\nu$ alone. Rather, it is a function of $\widehat\nu$ and one-step exit probabilities $P_\nu(\sigma,\tau)$, where $\sigma\in\widehat\Omega$ and $\tau\notin\widehat\Omega$. Moreover, this boundary form $\mathcal{Q}_{\nu}^{\widehat\Omega}(f;a)$ may be viewed as a penalty term, which penalizes deviations of $f$ from $a$, weighted by the one-step exit probabilities. Its role is analogous to that of the term $\|f\|_\infty^2$ appearing in weak Poincar\'{e} inequalities~\cite{rockner2001weak}, although here the penalty is confined to the boundary.
    
    Define the Dirichlet form with boundary by
    \[
    \mathcal E_\nu^\partial(f):=\inf_{a\in\R}\mathcal E_\nu^\partial(f;a),\quad\textnormal{where}\quad
    \mathcal E_\nu^\partial(f;a)
:=\mathcal E_{\widehat\nu}(f)+\mathcal Q_\nu^{\widehat\Omega}(f;a).
    \]
A crucial observation is that including the quadratic boundary term as above gives the uniform variance bound needed at $t=1$ in the localization argument: when $\nu$ is a \textit{product measure}, for every $a\in\R$,
    \begin{equation}\label{eq:conditioned:product:poincare}
 \Var_{\widehat\nu}(f)
\le\frac{\Var_\nu(f_a)}{\nu(\widehat\Omega)}
\le\frac{n\mathcal E_\nu(f_a)}{\nu(\widehat\Omega)}
=n\mathcal E_\nu^\partial(f;a).
    \end{equation}
    where the first inequality follows by law of total variance and the second follows since product measures satisfy tensorization of variance. Taking the infimum over $a\in\R$ gives $\Var_{\widehat\nu}(f)\le n\mathcal E_\nu^\partial(f)$. Thus, the conditioned product measure satisfies a uniform variance bound once the boundary term is included.

   For MLSI, we introduce an analogous \emph{entropic boundary term}. We work with the conditional entropy form: for a probability measure $\nu$ on $\Omega$, define
    \[
    \mathcal H_\nu(f):=\frac1n\sum_{i=1}^n\E_\nu\!\left[
        \Ent_{\nu(\cdot\mid\sigma_{-i})}(f)\right],\qquad f\geq0,
\]
where $\nu(\cdot\mid\sigma_{-i})$ denotes the conditional law of $\nu$ given all spins except the $i$'th spin. For $p\in[0,1]$ and $x,y\geq0$, write
    \[
        \ent_p(x,y):=px\log\bigg(\frac{x}{px+(1-p)y}\bigg)+(1-p)y\log\bigg(\frac{y}{px+(1-p)y}\bigg),
    \]
    where it is understood $0\log 0=0$. This is the entropy of a function taking values $x,y$ under probabilities $p$ and $1-p$. Thus, for $\sigma$ with $\nu(\sigma)>0$,
    \[
    \Ent_{\nu(\cdot\mid \sigma_{-i})}(f)=\ent_{\frac{\nu(\sigma)}{\nu(\sigma)+\nu(\sigma^{(i)})}}\Big(f(\sigma),f(\sigma^{(i)})\Big),
    \]
    where $\sigma^{(i)}$ is obtained by flipping the $i$th spin. The expectation $\E_{\nu}$ in $\mathcal H_\nu(f)$ averages this quantity over $\sigma$ distributed according to $\nu$. Moreover, since $\mathcal H_\nu(f)\le\mathcal E_\nu(f,\log f)$ for $f>0$ (see, e.g.,~\cite{CMT15}), proving the approximate tensorization of entropy $\Ent_\nu(f)\le Cn\mathcal H_\nu(f)$ suffices to establish MLSI. 
    
    For $f:\widehat\Omega\to[0,\infty)$ and $b\geq0$,
    define the entropic boundary term by
    \begin{equation}\label{eq:arbitrary-entropic-boundary}
       \mathcal B_\nu^{\widehat\Omega}(f)
        :=\inf_{b\ge0}\mathcal B_\nu^{\widehat\Omega}(f;b)\quad\textnormal{where}\quad \mathcal B_{\nu}^{\widehat\Omega}(f;b):=\frac1n\sum_{\substack{\sigma\in \widehat\Omega,\tau\notin \widehat\Omega\\\sigma\sim \tau}}
        \frac{\nu(\sigma)+\nu(\tau)}{\nu\big(\widehat\Omega\big)}
        \ent_{\frac{\nu(\sigma)}{\nu(\sigma)+\nu(\tau)}}\bigl(f(\sigma),b\bigr).
    \end{equation}
   Here and below, summands with $\nu(\sigma)=\nu(\tau)=0$ are understood to be $0$, and $\sigma\sim \tau$ denotes that $\tau$ is accessible from $\sigma$ by a single spin flip. Let
    \begin{align*}
         \mathcal H_\nu^\partial(f)
        :=\inf_{b\geq0}\mathcal H_\nu^\partial(f;b)\quad\textnormal{where}\quad\mathcal H_\nu^\partial(f;b)
        :=\mathcal H_{\widehat\nu}(f)+\mathcal B_\nu^{\widehat\Omega}(f;b).
    \end{align*}
    The following lemma gives the almost sure functional inequality for the stochastic localization process at $t=1$, the conditional product measure.
\begin{lemma}\label{lemma:general-boundary-decomp}
Let $\nu=\otimes_{i=1}^{n}\nu_i$ be a product probability measure on $\Omega=\{-1,+1\}^n$ and let $\widehat\Omega \subseteq \Omega$ be a subset with $\nu(\capOmega)>0$. Then for any $ f:\capOmega\to\R$, $\Var_{\widehat\nu}(f)\le n\mathcal E_\nu^\partial(f)$ holds. Moreover, for any $f: \capOmega\to [0,\infty)$,
\[
\Ent_{\widehat\nu}(f)\le n\mathcal H_\nu^\partial(f).
\]
\end{lemma}
\begin{proof}
The variance bound was proved in \eqref{eq:conditioned:product:poincare}. For entropy, consider $f:\widehat\Omega\to[0,\infty)$, $b\ge0$, and the constant extension $f_b:\Omega\to[0,\infty)$~\eqref{eq:constant:extension}. Note that
\[
\begin{split}
\mathcal H_\nu(f_b)
&=\frac{1}{2n}\sum_{(\sigma,\tau)\in \Omega^2, \sigma\sim \tau} (\nu(\sigma)+\nu(\tau))\ent_{\frac{\nu(\sigma)}{\nu(\sigma)+\nu(\tau)}}\big(f_b(\sigma),f_b(\tau)\big)
\\
&=\frac{1}{2n}\sum_{(\sigma,\tau)\in \widehat\Omega^2, \sigma\sim \tau} (\nu(\sigma)+\nu(\tau))\ent_{\frac{\nu(\sigma)}{\nu(\sigma)+\nu(\tau)}}\big(f(\sigma),f(\tau)\big)+\nu\big(\widehat\Omega\big)\mathcal{B}_{\nu}^{\widehat\Omega}(f;b).
\end{split}
\]
Since the first term in the right-hand side equals $\nu(\widehat\Omega)\mathcal{H}_{\widehat\nu}(f)$, we have $\mathcal{H}_{\nu}(f_b)=\nu(\widehat\Omega)\mathcal{H}_\nu^{\partial}(f;b)$. Thus, combining with total law of entropy and the tensorization of entropy for product measures~\cite{CMT15} yields
\[
\Ent_{\widehat\nu}(f)
\le\frac{\Ent_\nu(f_b)}{\nu\big(\widehat\Omega\big)}
\le\frac{n\mathcal H_\nu(f_b)}{\nu\big(\widehat\Omega\big)}
=n\mathcal H_\nu^\partial(f;b),
\]
and taking infimum over $b\geq 0$ concludes the proof.
\end{proof}

We remark that we could also consider a boundary term based on $\mathcal E_\nu(f,\log f)$, involving $(f(\sigma)-b)(\log f(\sigma)-\log b)$ for $f>0$ and $b>0$. The analog version of Lemma~\ref{lemma:general-boundary-decomp} (and also Proposition~\ref{prop:boundary-interaction-annealing} below) would still hold. However, the terms such as $\log f(\sigma)-\log b$ introduce complications for \emph{boundary absorption} step in Proposition~\ref{prop:defect-absorption} that we wish to avoid. For this step, $\mathcal{H}_{\nu}(f)$ is more convenient, which allows to use log-Sobolev inequalities for two point distributions~\cite{DSC96} to bound $\ent_p(x,y)$ in terms of $(\sqrt{x}-\sqrt{y})^2$. This allows to compare $\mathcal{B}(f)$ with $\mathcal{Q}(\sqrt{f})$; see Lemma~\ref{lemma:entropic-boundary-reduction}.

Motivated by Lemma~\ref{lemma:general-boundary-decomp}, we establish the following proposition in Section~\ref{sec:localization} using the localization scheme described above. This result generalizes \cite[Theorem 5.1]{ChenEldan}  to conditioned measures by including boundary terms and may be of independent interest.

\begin{proposition}\label{prop:boundary-interaction-annealing}
Let $\nu$ be a probability measure on $\Omega$, let $\widehat\Omega\subseteq\Omega$ satisfy $\nu(\widehat\Omega)>0$, and let $J\in\R^{n\times n}$ be symmetric positive definite. Write
$\widehat\nu:=\nu(\cdot\mid\widehat\Omega)$. For $t\in[0,1]$ and $h\in\R^n$, define
\begin{equation}\label{eq:nu:t:h}
    \nu_{t,h}(\sigma)\propto e^{-t\langle J\sigma,\sigma\rangle+\langle h,\sigma\rangle}\nu(\sigma),
    \qquad
    \widehat\nu_{t,h}:=\nu_{t,h}(\,\cdot\,|\widehat\Omega).
\end{equation}
Suppose that for some, $\alpha:[0,1]\to[0,\infty)$,
\begin{equation}\label{eq:annealing-covariance}
    \Cov(\widehat\nu_{t,h})\preceq\alpha(t)I_n,
    \qquad \forall t\in[0,1],\quad h\in\R^n.
\end{equation}
Then, we have the following.
\begin{itemize}
    \item[(a)] If for some $C_{\mathrm V}>0$, we have for every $h\in\R^n$ and every $f:\widehat\Omega\to\R$ that
    \[
     \Var_{\widehat\nu_{1,h}}(f)
        \le C_{\mathrm V}\mathcal E_{\nu_{1,h}}^\partial(f),
    \]
    then
    \begin{equation}\label{eq:annealing-variance-conclusion}
        \Var_{\widehat\nu}(f)\leq C_{\mathrm V}\exp\bigg( 2\|J\|_{\mathrm{op}}\int_0^1\alpha(t)dt \bigg)
        \mathcal E_\nu^\partial(f).
    \end{equation}

    \item[(b)] If for some $C_{\mathrm E}>0$, we have for every $h\in\R^n$ and every $f:\widehat\Omega\to[0,\infty)$ that
    \[
    \Ent_{\widehat\nu_{1,h}}(f)
        \le C_{\mathrm E}\mathcal H_{\nu_{1,h}}^\partial(f),
    \]
    then, 
    \begin{equation}\label{eq:annealing-entropy-conclusion}
        \Ent_{\widehat\nu}(f)\leq C_{\mathrm E}\exp\bigg( 2\|J\|_{\mathrm{op}}\int_0^1\alpha(t)dt \bigg)
        \mathcal H_\nu^\partial(f).
    \end{equation}
\end{itemize}
\end{proposition}
Combining the preceding proposition applied with $\nu=\mu_{J,h}$ and
Lemma~\ref{lemma:general-boundary-decomp} immediately gives the following corollary. The case $J\succeq0$ follows by replacing $J$ with $J+\delta I$, and letting $\delta\downarrow0$.
\begin{corollary}\label{cor:ising-restrictions}
Consider a non-empty subset $\widehat\Omega\subseteq\Omega$ and let $J\succeq0$.
Suppose that for some $\alpha:[0,1]\to[0,\infty)$,
\begin{equation}\label{eq:ising-restriction-covariance}
\Big\|\Cov\Big(\mu_{tJ,h}\big(\,\cdot\big|\widehat\Omega\big)\Big)\Big\|_{\op}\leq \alpha(t),\qquad \forall t\in[0,1],\quad h\in\mathbb R^n.
\end{equation}
Then, for every $h\in\R^n$, with $\widehat\nu:=\mu_{J,h}(\cdot\mid\widehat\Omega)$, we have
\[
\begin{split}
    \Var_{\widehat\nu}(f)&\leq n\exp\bigg( 2\|J\|_{\mathrm{op}}\int_0^1\alpha(t)dt \bigg)
       \mathcal E_{\mu_{J,h}}^\partial(f),\qquad
     f:\widehat\Omega\to\R, \\
    \Ent_{\widehat\nu}(f)&\leq n\exp\bigg( 2\|J\|_{\mathrm{op}}\int_0^1\alpha(t)dt \bigg)
       \mathcal H_{\mu_{J,h}}^\partial(f),\qquad
     f:\widehat\Omega\to[0,\infty).
\end{split}
\]
\end{corollary}

Next, we apply Corollary~\ref{cor:ising-restrictions} with $\widehat\Omega=\widehat\Omega_M$ to prove Theorem~\ref{theorem:core-func-ineq}. Under the sparse spectral width condition $8\omega_M(J)\le1-\eps$, Proposition~\ref{prop:sparse-covariance} verifies~\eqref{eq:ising-restriction-covariance} with $\alpha(t)=2/\eps$ (assuming w.l.o.g. $J\succeq 0$). The corollary therefore gives variance and entropy bounds in terms of $\mathcal E_{\mu_{J,h}}^\partial$ and $\mathcal H_{\mu_{J,h}}^\partial$.

To establish Poincar\'{e} and MLSI, it remains to compare these forms with $\mathcal E_{\mu_{J,h,M}}$ and $\mathcal H_{\mu_{J,h,M}}$, respectively. Their differences are precisely the quadratic and entropic boundary terms, $\mathcal Q_M(f)$ and $\mathcal B_M(f)$, where we abbreviate
\[
\begin{aligned}
\mathcal Q_M(f)&:=\inf_{a\in\mathbb R}\mathcal Q_M(f;a),\qquad \mathcal Q_M(f;a):=\mathcal Q^{\widehat\Omega_M}_{\mu_{J,h}}(f;a)\\
\mathcal B_M(f)&:=\inf_{b\geq0}\mathcal B_M(f;b),\qquad \mathcal B_M(f;b):=\mathcal B^{\widehat\Omega_M}_{\mu_{J,h}}(f;b).
\end{aligned}
\]
We control these terms by $\mathcal E_{\mu_{J,h,M}}(f)$ and $\mathcal H_{\mu_{J,h,M}}(f)$, respectively, up to a negligible error in terms of the variance and the entropy. This \emph{boundary absorption} step is where the drift conditions~\eqref{assumption:stability} and~\eqref{assumption:stability-entropy} enter: under the condition $\omega_M(J)<1/2$, they give uniform upward magnetization drift near the boundary of the core $\widehat\Omega_M$, i.e. $\{\sigma:\sum_i\sigma_i=M\}$ (cf. Lemma~\ref{prop:boundary-stability}). We prove the following proposition in Section~\ref{sec:boundary}.

\begin{proposition}
\label{prop:defect-absorption}
Suppose
$\omega_M(J)<1/2$, and let $R_M\equiv (n-M)/2\geq1$. There are universal constants $C\geq1$ and
$c>0$ such that the following hold.
\begin{itemize}
    \item[(a)] If \eqref{assumption:stability} holds, then every
    $f:\capOmega_M\to\R$ satisfies
    \[
        \mathcal Q_M(f)
        \leq C\mathcal E_{\mu_{J,h,M}}(f)
             +\frac{Ce^{-cR_M}}{n}\Var_{\mu_{J,h,M}}(f),
    \]
    and
    \[
        \Var_{\mu_{J,h,M}}(f)
        \leq Cn e^{8R_M\omega_M(J)}\mathcal E_{\mu_{J,h,M}}(f).
    \]

    \item[(b)] If \eqref{assumption:stability-entropy} holds, then every
    $f:\capOmega_M\to[0,\infty)$ satisfies
    \[
        \mathcal B_M(f)
        \leq C\mathcal H_{\mu_{J,h,M}}(f)
             +\frac{Ce^{-cR_M}}{n}\Ent_{\mu_{J,h,M}}(f),
    \]
    and
    \[
        \Ent_{\mu_{J,h,M}}(f)
        \leq Cn e^{8R_M\omega_M(J)}\mathcal H_{\mu_{J,h,M}}(f).
    \]
\end{itemize}
\end{proposition}
Assuming Propositions~\ref{prop:sparse-covariance}, \ref{prop:boundary-interaction-annealing}, and \ref{prop:defect-absorption}, Theorem~\ref{theorem:core-func-ineq} is immediate.
\begin{proof}[Proof of Theorem~\ref{theorem:core-func-ineq}]
For simplicity, write $R= R_M\geq 1$, $\nu=\mu_{J,h,M}$, and $\gamma=\|J\|_{\op}$. Let
\[
    J':=J-\lambda_{\min}(J)I_n\succeq0,
    \qquad \|J'\|_{\op}\leq2\gamma.
\]
Then $8\omega_M(tJ')=8t\omega_M(J)\leq1-\eps$ and  $\mu_{tJ',h}=\mu_{tJ,h}$ for all $t\in[0,1]$ and $h\in\R^n$. Thus, Proposition~\ref{prop:sparse-covariance} gives $\|\Cov(\mu_{tJ,h,M})\|_{\op}\leq 2/\eps$ for all $t\in [0,1]$ and $h\in \R^n$. Thus, Corollary~\ref{cor:ising-restrictions} applies with
$\alpha(t)\equiv2/\eps$. Combining it with the first estimate in Proposition~\ref{prop:defect-absorption}, under
\eqref{assumption:stability}, gives
\[
    \Var_\nu(f)
    \le Cne^{8\gamma/\eps}\mathcal E_\nu(f)
       +Ce^{8\gamma/\eps-cR}\Var_\nu(f),
\]
for universal $C,c>0$. If $R\ge C_0(1+\gamma/\eps)$ for a sufficiently
large universal $C_0$, absorb the last term into the left-hand side.
Otherwise, $8R\omega_M(J)\le R\lesssim1+\gamma/\eps$, and the second estimate in Proposition~\ref{prop:defect-absorption}-(a) gives the same bound.
Repeating the argument for entropy under
\eqref{assumption:stability-entropy}, we obtain
\begin{align*}
    \Var_\nu(f)
    &\le C_1ne^{C_2\gamma/\eps}\mathcal E_\nu(f),
    &&\text{under \eqref{assumption:stability}}\\
    \Ent_\nu(f)
    &\le C_1ne^{C_2\gamma/\eps}\mathcal H_\nu(f),
    &&\text{under \eqref{assumption:stability-entropy}},
\end{align*}
for universal $C_1,C_2>0$. The definitions of $\gap(\nu)$ and
$\rho_{\mathrm{LS}}(\nu)$, together with
$\mathcal H_\nu(f)\le\mathcal E_\nu(f,\log f)$ (see e.g.~\cite{CMT15}), complete the proof.
\end{proof}

\subsection{Phase ordering and posterior sampling for synchronization}
\label{subsec:z2-proof}
We now prove Theorem~\ref{theorem:z2-synchronization} from Theorem~\ref{theorem:core-func-ineq} and Theorem~\ref{theorem:lifting}, which give rapid mixing for the reflected dynamics in a suitable region. To obtain phase ordering for the unrestricted dynamics, we also need to show that it remains in this region long enough to mix. Coupled using the same Glauber updates, the two chains agree until the unrestricted chain leaves the region.

The following lemma provides this control through an upward magnetization drift. 
\begin{lemma}\label{lemma:basin-retention}
Let $(X_t)_{t\in\Z_{\geq0}}$ be a Markov chain on $\Omega$ whose
transition kernel $P$ changes at most one coordinate at each step.
Write $m(\sigma):=n^{-1}\sum_{i=1}^n\sigma_i$.
Let $-1<a<b<1$ and $\delta\in(0,1)$, and suppose that
\[
    n\bigl((Pm)(\sigma)-m(\sigma)\bigr)\geq\delta\quad\textnormal{for all $\sigma\in \Omega$ such that}\quad a\leq m(\sigma)\leq b.
\]
Consider the stopping times $\tau_0:=\inf\{t\geq0:m(X_t)<a\}$ and
$\tau_1:=\inf\{t\geq0:m(X_t)\geq b\}$.
There are constants $C,c,\gamma>0$, depending only on $a,b,\delta$,
such that, for all sufficiently large $n$ and every $x\in\Omega$
with $m(x)\geq(a+b)/2$,
\begin{equation}\label{ineq:positive-drift-retention}
    \mathbb P_x\left(
        \tau_1>Cn
        \ \text{or}\
        \tau_0\leq\lfloor e^{\gamma n}\rfloor
    \right)
    \leq e^{-cn}.
\end{equation}
\end{lemma}
\begin{proof}
The proof closely follows the stopped magnetization martingale argument in Gheissari, Sly, and Sohn~\cite[Lemmas~4.11 and~5.1]{GheissariSlySohn}, and we include the proof for completeness. Fix $x\in\Omega$ with $m(x)\geq(a+b)/2$.
Let $\mathcal F_t:=\sigma(X_0,\ldots,X_t)$,
 define the $\mathcal F_t$-martingale 
\[
    N_t:=m(X_t)-m(X_0)
        -\sum_{s=0}^{t-1}\bigl((Pm)(X_s)-m(X_s)\bigr).
\]
which has uniform increments $|N_{t+1}-N_t|\leq4/n$. Now, consider the stopping time $\tau:=\tau_0\wedge\tau_1$. For $s<\tau$, the drift assumption gives
$(Pm)(X_s)-m(X_s)\geq\delta/n$, so
\begin{equation*}
    N_{t\wedge\tau}
    \leq m(X_{t\wedge\tau})-m(x)
        -\frac{\delta}{n}(t\wedge\tau).
\end{equation*}
It follows that 
$\tau>T:=\lceil n(b-a)/\delta\rceil$ implies $N_T\leq-(b-a)/2$. Also, since $m(x)\geq (a+b)/2$, $\tau_0<\tau_1$ and $\tau_0\leq T$ imply
$N_{\tau_0}\leq-(b-a)/2$. Thus,
\begin{equation*}
\begin{split}
    \sup_{x:\,m(x)\geq(a+b)/2}
    \P_x(\tau_1>T\ \text{or}\ \tau_0<\tau_1)
    &\leq  \sup_{x:\,m(x)\geq(a+b)/2}\P_x\bigg(\min_{0\leq t\leq T}N_{t\wedge \tau}\leq -\frac{b-a}{2}\bigg)\\
    &\leq
    \exp\left(-\frac{(b-a)^2n^2}{128T}\right)\leq e^{-\kappa n}
\end{split}
\end{equation*}
where the second inequality holds by maximal Azuma-Hoeffding inequality (see e.g. \cite[Theorem 3.2.1]{roch2024modern}), applied to the stopped martingale $(-N_{t\wedge\tau})_{t\ge0}$, whose increments are bounded in absolute value by $4/n$. The last inequality holds
for some $\kappa=\kappa(a,b,\delta)>0$.

Each downward crossing of $b$ lands at a configuration $y$ with
$m(y)\in[b-2/n,b)$, hence $m(y)\geq(a+b)/2$ for large $n$.
The preceding bound shows that each such excursion reaches below
$a$ before returning to $b$ with probability at most $e^{-\kappa n}$.
A union bound over at most $\lfloor e^{\kappa n/3}\rfloor$ crossings, together with the strong
Markov property at $\tau_1$, gives
\[
    \sup_{x:\,m(x)\geq(a+b)/2}
    \P_x(\tau_1>T\ \text{or}\ \tau_0\leq \lfloor e^{\kappa n/3}\rfloor)
    \leq(\lfloor e^{\kappa n/3}\rfloor+1)e^{-\kappa n}.
\]
Using $T\leq Cn$, with $C=2(b-a)/\delta$, proves the claim. 
\end{proof}
Recall the posterior distribution $\pi_Y$ in \eqref{eq:posterior:Z2}. Under the change of variables $\tau_i=\theta_i\sigma_i$, we have $\pi_Y(\diag(\theta)\tau)=\mu_{J_0,0}(\tau)$, where
\[
J_{0}=\frac{\lambda^2}{2n}\1\1^\top+\frac{\lambda}{2} \diag(\theta)W\diag(\theta).
\]
The following result collects the estimates needed to apply Theorems~\ref{theorem:lifting} to $\mu_{J_0,0}$. Its proof is given in Section~\ref{subsec:z2-estimates}.

\begin{proposition}\label{prop:z2-inputs}
There are universal constants $C',\lambda_0>0$ such that
the following holds for every fixed $\lambda\geq\lambda_0$.
Put
\[
    m_0:=\frac{C'}{2\lambda},
    \qquad
    R:=\left\lfloor\frac{n}{\lambda^3}\right\rfloor,
    \qquad
    M:=n-2R.
\]
Let $P$ denote the discrete time Glauber transition kernel for $\mu_{J_0,0}$. With probability $1-o(1)$ over $(\theta,W)$,
we have $8\omega_M(J_0)\leq1/2$, and that
\eqref{assumption:stability-entropy} holds at cutoff $M$,
and
\begin{equation}\label{eq:z2-input-drift}
    n\bigl((Pm)(\tau)-m(\tau)\bigr)
    \geq\frac{1-m(\tau)}2,
    \quad\textnormal{for all $\tau$ such that}\quad m_0\leq m(\tau)<M/n.
\end{equation}
\end{proposition}

\begin{proof}[Proof of Theorem~\ref{theorem:z2-synchronization}]
We use the notations $m_0,R,M,$ and $P$ as in Proposition~\ref{prop:z2-inputs}, and work on its high-probability event together with $\|W\|_{\op}\leq C$, where the latter holds w.h.p. since $W$ is a GOE (see e.g.~\cite[Theorem 4.4.3]{Vershynin2026}). Let $S$ be the smallest integer of the same parity as $n$ with $S\geq m_0n$, and set
\[
    \widehat\Omega_{\theta,S}:=\{\sigma:m_\theta(\sigma)\geq m_0\},
    \qquad
    \widehat\pi_S:=\pi_Y(\,\cdot\mid\widehat\Omega_{\theta,S}).
\]
Write $\widehat P$ for the discrete time Glauber transition kernel on $\widehat\pi_S$. The change of variables $\tau_i=\theta_i\sigma_i$ identifies $\widehat\pi_S$ and $\widehat P$ with $\mu_{J_0,0,S}$ and discrete time Glauber dynamics on $\mu_{J_0,0,S}$. In these coordinates, we apply Theorem~\ref{theorem:lifting}-(b) and Lemma~\ref{lemma:basin-retention}. 

We first verify the conditions of Theorem~\ref{theorem:lifting}-(b). Since we work on the w.h.p. event of Proposition~\ref{prop:z2-inputs}, $8\omega_M(J_0)\leq 1/2$ holds and \eqref{assumption:stability-entropy} is satisfied. Moreover, since $n((Pm)(\tau)-m(\tau))=2(q_{+}(\tau)-q_{-}(\tau))$, \eqref{eq:z2-input-drift} verifies
\eqref{assumption:tail-stability} with $\zeta=1/2$.
Finally, for sufficiently large $\lambda_0$ and $n$, we have
$0\leq S<M$ and $M\leq(1-\lambda^{-3})n$, and since
$\|J_0\|_{\op}\leq\lambda^2/2+C\lambda/2$, it follows that
\begin{equation}\label{eq:z2-input-information}
    \max_{\sigma\in\widehat\Omega_{\theta,S}}
    \log\bigg(\frac1{\widehat\pi_S(\sigma)}\bigg)
    \leq n\bigl(\log2+2\|J_0\|_{\op}\bigr)
    \leq C_\lambda n.
\end{equation}
Thus, \eqref{assumption:H4} holds with $\eta=C_{\lambda}$. Therefore, Theorem~\ref{theorem:lifting}-(b) with $\eps=1/2$ and
$\vartheta=\lambda^{-3}$ yields
$\rho_{\mathrm{LS}}(\widehat\pi_S)\geq1/(C_\lambda n)$.
Together with \eqref{eq:z2-input-information} and
\eqref{eq:mixing-time-bounds}, this implies the mixing bound
\[
    \sup_{x\in\widehat\Omega_{\theta,S}}
    \|\widehat P^t(x,\cdot)-\widehat\pi_S\|_{\TV}
    \leq n^{-12},
    \qquad t\geq C_\lambda n\log n.
\]
Increasing $\lambda_0$ ensures that $3m_0\leq1/2<M/n$.
On $m_0\leq m\leq3m_0$, \eqref{eq:z2-input-drift} gives
$n(Pm-m)\geq1/4$. Lemma~\ref{lemma:basin-retention},
applied with $a=m_0$, $b=3m_0$, and $\delta=1/4$,
therefore bounds the probability of leaving
$\widehat\Omega_{\theta,S}$ before $e^{\gamma_\lambda n}$
by $e^{-c_\lambda n}$, uniformly over initial configurations
with $m_\theta(x)\geq2m_0$. Coupling the unrestricted chain $(X_t)_{t\geq 0}$ and the
reflected chain using the same updates gives
\begin{equation}\label{eq:z2-basin-coupling}
    \sup_{x:\,m_\theta(x)\geq2m_0}
    \|\P_x(X_t\in\cdot)-\widehat\pi_S\|_{\TV}
    \leq n^{-12}+e^{-c_\lambda n},\quad\textnormal{for all}\quad C_\lambda n\log n\leq t\leq e^{\gamma_\lambda n}.
\end{equation}
To compare $\widehat{\pi}_S$ with $\pi_Y(\cdot\mid \Omega_{\theta}^{+})$, observe that,
for any $\sigma\in \Omega$,
\[
    \log\bigg(\frac{\pi_Y(\theta)}{\pi_Y(\sigma)}\bigg)=n\left(\frac{\lambda^2}{2}\big(1-m_\theta(\sigma)^2\big)+\frac{\lambda}{2n}\big(\langle \theta,W\theta\rangle -\langle \sigma, W\sigma\rangle \big)\right)
    \geq n\left(\frac{\lambda^2}{2}\big(1-m_\theta(\sigma)^2\big)-C\lambda\right).
\]
In particular, if $|m_{\theta}(\sigma)|\leq m_0$, then the right-hand side is at least $n\lambda^2/4$ after increasing $\lambda_0$. Consequently,
\[
    \|\widehat\pi_S-\pi_Y(\,\cdot\mid\Omega_\theta^+)\|_{\TV}=\frac{\pi_Y\big(\{\sigma:0\leq m_\theta(\sigma)<m_0\}\big)}{\pi_Y\big(\Omega_{\theta}^{+}\big)}\leq 2^{n+1}e^{-n\lambda^2/4},
\]
where the last inequality holds by a union bound and $\pi_Y(\Omega_{\theta}^{+})\geq 1/2$ by spin-flip symmetry. Combining this with \eqref{eq:z2-basin-coupling} proves
\eqref{eq:z2-phase-ordering}, since $2m_0=C'/\lambda$,
after adjusting $C_0=C_0(\lambda)$.

For spectral initialization, let $v=v(Y)$ be a unit leading
eigenvector. Let $s\in \{-1,+1\}$ such that $s\langle \theta, v\rangle =|\langle \theta, v\rangle |$, and set $A=\{i: \sign(v_i)=s\theta_i\}$. Then,
\[
    |m_\theta(\sign(v))|\geq \frac{s\langle \theta, \sign(v)\rangle}{n}
    =\frac{2|A|}{n}-1\geq \frac{2|\langle \theta,v\rangle|^2}{n}-1,
\]
where in the last inequality, we used $|\langle \theta, v\rangle|\leq \sum_{i\in A}|v_i|$ by definition of $s$ and $A$, and Cauchy-Schwartz. By~\cite[Theorem~2.2]{benaych2011eigenvalues}, we know that for every fixed $\lambda>1$, $|\langle\theta/\sqrt n,v\rangle|^2$ converges in probability to $1-\lambda^{-2}$. Thus, after increasing $\lambda_0$, we have $|m_{\theta}(\sign(v))|\geq 2m_0$ with probability $1-o(1)$. Moreover, by spin-flip symmetry,
\[
    \frac12\pi_Y(\,\cdot\mid\widehat\Omega_{\theta,S})
    +\frac12\pi_Y(\,\cdot\mid-\widehat\Omega_{\theta,S})
    =\pi_Y(\,\cdot\mid |m_\theta|\geq m_0).
\]
Thus, \eqref{eq:z2-basin-coupling} and an independent fair sign
$\xi$ yield
\[
    \|\P_{\sign(v)}(\xi X_t\in\cdot)-\pi_Y\|_{\TV}
    \leq n^{-12}+e^{-c_\lambda n}+2^{n+1}e^{-n\lambda^2/4}
    \leq n^{-10}
\]
throughout $C_0n\log n\leq t\leq e^{n/C_0}$, after adjusting $C_0=C_0(\lambda)$, which concludes the proof.
\end{proof}

\section{Spectral independence of the core-conditioned measure}
\label{sec:spectral-independence}
This section proves Proposition~\ref{prop:sparse-covariance}. Fix $v\in \R^n$ and let $X,Y\stackrel{iid}{\sim}\mu_{J,h,M}$. Then,
\[
v^{\top}\Cov(\mu_{J,h,M})v=\Var\big(v^\top X\big)=\frac{1}{2}\E\Big[\big(v^\top(X-Y)\big)^2\Big].
\]
Define the (random) agreement and disagreement sets of spins of $X,Y$ by
\[
A:=\left\{i\in [n]:X_i=Y_i\right\}, \qquad B:=\left\{i\in[n]:X_i\neq Y_i\right\}.
\]
For a set $I\subset[n]$, write $X_I:=(X_i)_{i\in I}\in \R^{I}$, and use analogous notation for other vectors. Since the spins take values in $\{+1,-1\}$,
\[
X_A=Y_A\in \{+1,-1\}^A, \qquad X_B=-Y_B\in \{+1,-1\}^B.
\]
Hence $v^\top (X-Y)=2v_B^\top X_B$, and the tower property gives
\begin{equation}
\label{eq:duplicate:identity}
v^{\top}\Cov(\mu_{J,h,M})v=2\E\Big[\E\Big[\big(v_B^\top X_B\big)^2\Big| B, X_A\Big]\Big].
\end{equation}
We next calculate the conditional law of $X_B$ given $B,X_A$.
The key observation is that, for $x,y\in\{\pm1\}^n$
with agreement set $A$ and disagreement set $B$,
\[
\langle h,x\rangle+\langle x,Jx\rangle +\langle h,y\rangle +\langle y,Jy\rangle = 2\langle h_A,x_A\rangle + 2\langle x_A,J_{A,A} x_A\rangle+ 2\langle x_B, J_{B,B} x_B\rangle.
\]
That is, the external field on $B$ and the interactions between
$A$ and $B$ cancel between the two replicas. This identity using ``duplicate'' variables dates back to the work by Percus~\cite{percus1975correlation} and Lebowitz~\cite{lebowitz1974ghs} (see Eq.~(2.6) therein) proving correlation inequalities for Ising models. 

Note that the two constraints
$(x_A,x_B)\in\widehat{\Omega}_M$ and
$(x_A,-x_B)\in\widehat{\Omega}_M$ are equivalent to
\[
    \bigg|\sum_{i\in B}x_i\bigg|
    \le \sum_{i\in A}x_i-M.
\]
Given $B$, $X_A=x_A$, and $X_B=x_B$, we necessarily have $Y=(x_A,-x_B)$. Thus, for $x_B\in \{+1,-1\}^B$,
\begin{equation}\label{eq:conditional:law:XB:given:XA}
\begin{split}
    \mathbb P\big(X_B=x_B\big| B,X_A=x_A\big)
    &=\frac{\P(X_A=x_A, X_B=x_B\mid B)}{\sum_{x'_B\in \{+1,-1\}^B}\P(X_A=x_A, X_B=x'_B\mid B)}
    \\
    &\propto
    \exp\Big(2\big\langle x_B,J_{B,B}x_B\big\rangle\Big)
    \one_{\left\{
        \left|\sum_{i\in B}x_i\right|
        \le \sum_{i\in A}x_i-M
    \right\}}.
\end{split}
\end{equation}
Thus, conditional on $B$ and $X_A=x_A$, the law of $X_B$
is an Ising measure with interaction matrix $2J_{B,B}$ and
zero external field, restricted to  $\left|\sum_{i\in B}X_i\right|\le \sum_{i\in A}x_i-M$.

To bound this covariance, we use the following extension of Mikulincer and Sohn~\cite[Lemma~1]{mikulincer2025fast}.
The argument comes from the work of
Bauerschmidt and Bodineau~\cite{BB19}, also used by
Chen and Eldan~\cite[Lemma~5.2]{ChenEldan}. 

\begin{lemma}\label{lem:covaraince:transfer:MS}
    Let $\rho$ be a probability measure on $\{+1,-1\}^m$, and for $J\in \R^{m\times m}$ and $h\in \R^m$, write
\[
    \rho_{J,h}(x)\propto
    e^{\langle Jx,x\rangle+\langle h,x\rangle}\rho(x).
\]
Suppose that, for some $L>0$,
\[
    \sup_{h\in\mathbb R^m}
    \|\operatorname{Cov}(\rho_{0,h})\|_{\mathrm{op}}\leq L,
\]
and $J\succeq0$ with $2L\|J\|_{\mathrm{op}}<1$. Then, 
\[
    \sup_{h\in\mathbb R^m}
    \|\operatorname{Cov}(\rho_{J,h})\|_{\mathrm{op}}
    \leq\frac{L}{1-2L\|J\|_{\mathrm{op}}}.
\]

\end{lemma}
\begin{proof}
    We sketch the proof since the argument is identical to the proof of \cite[Lemma 1]{mikulincer2025fast}. By adding a multiple of the identity matrix, we assume w.l.o.g. that $0\prec J \prec \frac{1}{2L}I_m$. By Gaussian integration, one can check that for $x\in \{+1,-1\}^m$,
    \[
    \rho_{J,h}(x)\propto \int_{\R^n}e^{-f(y)}\rho_{0,h+y}(x)\, dy,\qquad f(y)=\frac{1}{2}\big\langle y, (2J)^{-1}y \rangle-\log Z_{h+y},
    \]
    where $Z_{h+y}=\sum_{x\in \{+1,-1\}^m}\exp(\langle x,h+y\rangle)\rho(x)$. By the stated assumption, we have
    \[
    \nabla^2 f(y)=(2J)^{-1}-\Cov(\rho_{0,h+y})\succeq \left((\|2J\|_{\op})^{-1}-L\right) I_m\succ 0.
    \]
    This log-concavity of the measure $e^{-f(y)}dy$, together with Brascamp-Lieb inequality~\cite[Theorem 4.1]{BL76} and total law of variance, gives the desired conclusion.
\end{proof}
\noindent
Combining Lemma~\ref{lem:covaraince:transfer:MS} with the theory of negative dependence~\cite{pemantle2000towards, BBL09} gives the following.
\begin{lemma}\label{lem:covariance:transfer}
    Let $J\in \R^{m\times m}$ be a symmetric matrix with $\lambda_{\max}(J)-\lambda_{\min}(J)\leq \frac{1-\eps}{4}$ for some $\eps>0$. For $h\in \R^m$ and $I=[a,b]$, where $a,b\in \mathbb{Z}$, consider the probability measure $\nu$ on $\{+1,-1\}^m$ defined by
    \[
    \nu(x)\propto \exp\big(\langle x, Jx\rangle + \langle h, x\rangle\big)\one_{\{\sum_{i=1}^{m}x_i\in I\}}.
    \]
    Then
    \[
    \left\|\Cov(\nu)\right\|_{\op}\leq \frac{2}{\eps}.
    \]
\end{lemma}
\begin{proof}
Let $\rho$ be the uniform measure on $\{x\in \{+1,-1\}^m:\sum_{i=1}^{m}x_i\in I\}$. Then $\rho_{0,h}(x)\propto e^{\langle h,x\rangle}\rho(x)$ is the law of independent $\{+1,-1\}$ valued variables conditioned on their sum
belonging to $I$. It is known (see e.g.~\cite[Theorem 3.7]{BBL09}) that such measures are negatively correlated. Thus, $\Cov(\rho_{0,h})$ is positive semi-definite with non-positive off-diagonal entries and diagonal entries at most $1$, which implies that
    \[
\left\|\Cov(\rho_{0,h})\right\|_{\op}\leq 2.
    \]
    Also, by shifting by a multiple of identity, we may assume w.l.o.g. $0\preceq J\preceq \frac{1-\eps}{4}I_m$.
    Therefore, applying Lemma~\ref{lem:covaraince:transfer:MS} with $L=2$ gives $\|\Cov(\nu)\|_{\op}
\le \frac{2}{1-4\|J\|_{\op}}
\le \frac{2}{\eps}$.
\end{proof}
We apply Lemma~\ref{lem:covariance:transfer} with $\nu$ the conditional law of $X_B$ given $B,X_A$ to prove Proposition~\ref{prop:sparse-covariance}.
\begin{proof}[Proof of Proposition~\ref{prop:sparse-covariance}]
  Fix a unit vector $v\in\R^n$. By
\eqref{eq:conditional:law:XB:given:XA}, conditional on $B,X_A$,
the law of $X_B$ is an Ising measure with interaction matrix
$2J_{B,B}$ and zero external field, restricted to $|\sum_{i\in B}X_i|
\le \sum_{i\in A}X_i-M.$ This law is invariant under $X_B\mapsto-X_B$, so
$\E[X_B\mid B,X_A]=0$.

Since each disagreement contributes one negative spin to the
pair $(X,Y)$, we have 
\[
|B|\le2R_M\le n-M
\]
Thus, we have almost surely
\[
2\bigl(\lambda_{\max}(J_{B,B})-\lambda_{\min}(J_{B,B})\bigr)
\le 2\omega_M(J)\le\frac{1-\eps}{4}.
\]
Lemma~\ref{lem:covariance:transfer} therefore gives
\[
\|\Cov(X_B\mid B,X_A)\|_{\op}\le\frac{2}{\eps}.
\]
Combining with \eqref{eq:duplicate:identity} and $\E[X_B\mid B,X_A]=0$, we have
\[
\begin{aligned}
v^\top\Cov(\mu_{J,h,M})v=2\E\Big[\E\big[(v_B^\top X_B)^2\mid B,X_A\big]\Big]\le\frac{4}{\eps}\E\Big[\|v_B\|_2^2\Big].
\end{aligned}
\]
Finally, independence of $X,Y$ gives $\P(i\in B)
=2\P(X_i=1)\bigl(1-\P(X_i=1)\bigr)\le\frac12.$ Hence
$\E[\|v_B\|_2^2]=\sum_i v_i^2\P(i\in B)\le1/2$,
which proves the desired bound.
\end{proof}

\section{Stochastic localization with a boundary form}\label{sec:localization}

This section proves Proposition~\ref{prop:boundary-interaction-annealing} using the stochastic localization scheme~\eqref{eq:def:SL} of Eldan~\cite{eldan2013thin}. Throughout, we fix a probability measure $\nu$ on $\Omega\equiv \{-1,+1\}^n$, a symmetric positive definite matrix $J\in \R^{n\times n}$, and a subset $\widehat{\Omega}\subseteq \Omega$ such that $\nu(\widehat\Omega)>0$. 

In our conditioning argument below, it is useful to adopt the following information-theoretic view of stochastic localization by El Alaoui and Montanari~\cite{elalaoui2022information}: when $C_t\equiv C$ in the SDE~\eqref{eq:def:SL}, the random measure $\nu_t$ can be represented as the posterior distribution of $X\sim\nu$ given the noisy observation $Y_t=tCX+B_t$. When the state space is discrete, we have the following.
\begin{theorem}[{\cite[Theorem~2]{elalaoui2022information}}]
\label{theorem:gaussian-localization}
Let $\nu$ be a probability measure on $\Omega$ and let $C\in \R^{n\times n}$ be a symmetric
positive definite matrix. Under $\P$, let $X\sim\nu$ be independent of
standard Brownian motion $(B_t)_{t\geq0}$ in $\R^n$, and set
\[
    Y_t:=tCX+B_t,
    \qquad \mathcal F_t:=\sigma(Y_s:0\leq s\leq t).
\]
Then the conditional distribution $\nu_t(x):=\P(X=x\mid\mathcal F_t)$
satisfies
\begin{equation}\label{eq:full-posterior-localization}
    \nu_t(x)\propto
    \exp\bigg(\langle CY_t,x\rangle-\frac t2\|Cx\|_2^2\bigg)\nu(x).
\end{equation}
Moreover, letting $b(\nu_t)\equiv \sum_{x\in \Omega}x\nu_t(x)$ denote the barycenter, $ \widetilde B_t:=Y_t-\int_0^t Cb(\nu_s)ds$ is an $(\mathcal F_t)$-Brownian motion under $\P$, and
\[
    d\nu_t(x)=\nu_t(x)\langle x-b(\nu_t),C\,d\widetilde B_t\rangle,
    \qquad \nu_0=\nu.
\]
\end{theorem}
To prove Proposition~\ref{prop:boundary-interaction-annealing}, we take $C=(2J)^{1/2}$ in Theorem~\ref{theorem:gaussian-localization}. With $\nu_t$ defined in \eqref{eq:full-posterior-localization} and $\nu_{t,h}$ in \eqref{eq:nu:t:h}, we have
\[
\nu_t=\nu_{t,h_t},\quad\textnormal{where}\quad h_t=(2J)^{1/2}Y_t.
\]
We consider the process $(\nu_t)_{t\in [0,1]}$ conditioned on $\widehat\Omega$. Define
\[
     \widehat\nu_t:=\nu_t(\,\cdot\mid\widehat\Omega),\qquad \widehat{\P}:=\P(\,\cdot\mid X\in\widehat\Omega),
\]
and write $\widehat{\E}$ for expectation under $\widehat{\P}$. These conditional measures are well-defined, since $\nu(\widehat\Omega)>0$ and \eqref{eq:full-posterior-localization} imply $\nu_t(\widehat\Omega)>0$ almost surely for all $t\in [0,1]$. In particular, $\widehat\nu_t=\widehat\nu_{t,h_t}$.

A first observation is that $(\widehat\nu_t)_{t\in [0,1]}$ is itself a stochastic localization process under $\widehat\P$. Indeed, Bayes rule gives
\begin{equation}\label{eq:capped-posterior-localization}
    \widehat\nu_t(x)
    =\frac{\P(X=x\mid \cF_t)\one_{x\in \widehat\Omega}}{\P(X\in \widehat\Omega\mid \cF_t)}
    =\widehat{\P}(X=x\mid\mathcal F_t).
\end{equation}
Note that under $\widehat{\P}$, $X$ has law
$\widehat\nu\equiv \nu(\,\cdot\mid\widehat\Omega)$ and remains independent
of $B$. Thus, applying Theorem~\ref{theorem:gaussian-localization} with prior
$\widehat\nu$ yields
\begin{equation}\label{eq:restricted-localization}
    d\widehat\nu_t(x)
    =\widehat\nu_t(x)
       \big\langle x-b(\widehat\nu_t),C\,d\widehat B_t\big\rangle,
    \qquad \widehat\nu_0=\widehat\nu,
\end{equation}
where $(\widehat B_t)_{t\geq 0}$ is an $(\mathcal F_t)$-Brownian motion under
$\widehat{\P}$.

Using this observation, we prove that the Dirichlet and conditional entropy forms with boundary terms are also supermartingales under $\widehat{\P}$. This generalizes a crucial property used in stochastic localization arguments, established without boundary terms in~\cite[Lemma~9]{EKZ} and~\cite[Proposition~4.3]{ChenEldan}.

\begin{proposition}\label{prop:defect-supermartingale}
For $f:\widehat\Omega\to \R$, the process $\bigl(\mathcal E_{\nu_t}^\partial(f)\bigr)_{0\leq t\leq1}$ is a $\mathcal{F}_t$-supermartingale under $\widehat{\P}$. The same holds for $\bigl(\mathcal H_{\nu_t}^\partial(f)\bigr)_{0\leq t\leq1}$ for every $f:\widehat\Omega\to[0,\infty)$.
\end{proposition}

\begin{proof}
By continuity, the infima defining $\mathcal{E}_{\nu_t}^{\partial}(f)\equiv \inf_{a\in \R}\mathcal{E}_{\nu_t}^{\partial}(f;a)$, and $\mathcal{H}_{\nu_t}^{\partial}(f)\equiv \inf_{b\geq 0}\mathcal{H}_{\nu_t}^{\partial}(f;b)$ may be taken over $a\in \Q$ and $b\in \Q_{\geq 0}$, respectively. It therefore suffices to prove that $\mathcal E_{\nu_t}^{\partial}(f;a)$ and $\mathcal H_{\nu_t}^{\partial}(f;b)$ are supermartingales for every fixed $a\in\R$ and $b\geq0$, respectively.

Recall that these forms are not determined by $\widehat\nu_t$ alone:
they also depend on the transition probabilities $P_{\nu_t}(x,y)$
with $x\in\widehat\Omega$ and $y\notin\widehat\Omega$.
We therefore consider the normalized weights
\[
w_t(x):= \frac{\nu_t(x)}{\nu_t(\widehat\Omega)},\qquad x\in \Omega.
\]
Then, $w_t(x)=\widehat\nu_t(x)$ for $x\in \widehat\Omega$, so $w_t(x)$ is a martingale for $x\in \widehat\Omega$. We first extend this property to $x\notin \widehat\Omega$: indeed, for $s\leq t$,
\[
    \widehat{\E}\big[
       w_t(x)\big|\mathcal F_s
    \big]=\frac{1}{\nu_s(\widehat\Omega)}\E\left[
        \one_{\{X\in\widehat\Omega\}}
       w_t(x)\,\middle|\,\mathcal F_s
    \right]=\frac1{\nu_s(\widehat\Omega)}\E[\nu_t(x)\mid\mathcal F_s]
    =w_s(x).
\]
The first equality follows from Bayes rule, the second from $\E[\one_{{X\in\widehat\Omega}}|\mathcal F_t]=\nu_t(\widehat\Omega)$ and the tower property,
and the last from the martingale property of $\nu_t(x)$ under $\P$.

Note that $\mathcal{E}_{\nu_t}^{\partial}(f;a)$ and $\mathcal{H}_{\nu_t}^{\partial}(f;b)$ can be expressed with $(w_t(x))_{x\in \Omega}$ as follows. By definition of $\mathcal{E}_{\nu}^{\partial}(f;a)$ and $P_{\nu}(\sigma,\tau)$, we have
\[
\mathcal{E}_{\nu_t}^{\partial}(f;a)=\frac{1}{2n}\sum_{\sigma,\tau\in \widehat\Omega,\sigma\sim\tau}\frac{w_t(\sigma)w_t(\tau)}{w_t(\sigma)+w_t(\tau)}\big(f(\sigma)-f(\tau)\big)^2+\frac{1}{n}\sum_{\sigma\in \widehat\Omega,\tau\notin \widehat\Omega, \sigma\sim \tau}\frac{w_t(\sigma)w_t(\tau)}{w_t(\sigma)+w_t(\tau)}\big(f(\sigma)-a\big)^2.
\]
Similarly,
\[
\begin{split}
\mathcal{H}_{\nu_t}^{\partial}(f;b)
&=\frac{1}{2n}\sum_{\sigma,\tau\in \widehat\Omega, \sigma\sim \tau} \big(w_t(\sigma)+w_t(\tau)\big)\ent_{\frac{w_t(\sigma)}{w_t(\sigma)+w_t(\tau)}} \big(f(\sigma),f(\tau)\big)\\
&\qquad\quad+\frac{1}{n}\sum_{\sigma\in \widehat\Omega,\tau\notin \widehat\Omega, \sigma\sim \tau} \big(w_t(\sigma)+w_t(\tau)\big)\ent_{\frac{w_t(\sigma)}{w_t(\sigma)+w_t(\tau)}} \big(f(\sigma),b\big)
\end{split}
\]
Here, it is understood that the summands with $w_t(\sigma)+w_t(\tau)=0$ are $0$. Recall that $(w_t(\sigma))_{\sigma\in \Omega}$ is a non-negative martingale.  For any two non-negative martingales $M_t,N_t$ and fixed $u,v\geq0$, both $\frac{M_tN_t}{M_t+N_t}$ and
\[
(M_t+N_t)\ent_{\frac{M_t}{M_t+N_t}}(u,v)=uM_t\log u+vN_t\log v+(uM_t+vN_t)\log\bigg(\frac{M_t+N_t}{uM_t+vN_t}\bigg)
\]
are supermartingales since $(x,y)\mapsto \frac{xy}{x+y}$ and $(x,y)\mapsto x\log (y/x)$ are concave functions (see also~\cite[Fact A.2]{ChenEldan}). Both supermartingales are non-negative and bounded by a constant depending only on $u,v$ times $M_t+N_t$, so they are integrable. Since $\mathcal E_{\nu_t}^{\partial}(f;a)$ and $\mathcal H_{\nu_t}^{\partial}(f;b)$ are sums of such supermartingales, they are also supermartingales, which concludes the proof.
\end{proof}
The following result of~\cite{ChenEldan} shows that covariance bounds uniform over linear tilts give \emph{approximate conservation of variance and entropy} in expectation along the localization process.

\begin{proposition}[{\cite[Proposition~3.19 and Lemma~3.20]{ChenEldan}}]\label{fact:EC-conservation}
Let $(\nu_t)_{0\leq t\leq1}$ be a stochastic localization process satisfying~\eqref{eq:def:SL} with initial measure $\nu_0=\nu$ and an adapted process of $n\times n$ positive semi-definite matrices $(C_t)_{t\geq 0}$. 
For $h\in\R^n$, write
$(\nu_t)_h(x)\propto e^{\langle h,x\rangle}\nu_t(x)$.
Suppose a deterministic function $r:[0,1]\to[0,\infty)$
satisfies
\[
    \|C_t\Cov((\nu_t)_h)C_t\|_{\op}\leq r(t)
\]
almost surely for all $t\in[0,1]$ and $h\in\R^n$.
Then
\begin{align*}
    \E[\Var_{\nu_1}(f)]
    &\geq e^{-\int_0^1r(t)dt}\Var_{\nu}(f),
    &&f:\Omega\to\R,\\
   \E[\Ent_{\nu_1}(f)]
    &\geq e^{-\int_0^1r(t)dt}\Ent_{\nu}(f),
    &&f:\Omega\to[0,\infty).
\end{align*}
\end{proposition}

\begin{proof}[Proof of Proposition~\ref{prop:boundary-interaction-annealing}]
For all $t\in [0,1]$, let $C_t\equiv C=(2J)^{1/2}$. We apply Proposition~\ref{fact:EC-conservation} to $(\widehat{\nu}_t)_{t\in[0,1]}$ under $\widehat{\P}$, which is a stochastic localization process by~\eqref{eq:restricted-localization}. Since $\widehat\nu_t=\widehat\nu_{t,h_t}$ with $\widehat\nu_{t,h}$ defined in~\eqref{eq:nu:t:h}, we have $(\widehat\nu_t)_h=\widehat\nu_{t,h_t+h}$. Thus, assumption~\eqref{eq:annealing-covariance} gives
\[
    \|C\Cov((\widehat\nu_t)_h)C\|_{\op}
    \leq2\|J\|_{\op}\alpha(t).
\]
Set $\gamma:=2\|J\|_{\op}\int_0^1\alpha(t)dt$. Proposition~\ref{fact:EC-conservation} thus yields
\[
    \Var_{\widehat\nu}(f)
    \leq e^\gamma\widehat{\E}[\Var_{\widehat\nu_1}(f)]\leq C_{\mathrm V}e^\gamma
          \widehat{\E}[\mathcal E_{\nu_1}^\partial(f)],
\]
where the last inequality holds by our assumption $\Var_{\widehat{\nu}_{1,h}}(f)\leq C_{\mathrm V}\mathcal{E}_{\nu_{1,h}}^{\partial}(f)$ along with $\widehat{\nu}_{1}=\widehat{\nu}_{1,h_1}$ and $\nu_1=\nu_{1,h_1}$. Since $(\mathcal{E}_{\nu_t}^{\partial}(f))_{t\in [0,1]}$ is a supermartingale under $\widehat{\P}$ by Proposition~\ref{prop:defect-supermartingale} and $\nu_0=\nu$, we have
\[
  \widehat{\E}[\mathcal E_{\nu_1}^\partial(f)]\leq \mathcal{E}_{\nu}^{\partial}(f),
\]
Combining these inequalities prove part~(a).

For part~(b), the same argument using Propositions~\ref{prop:defect-supermartingale} and~\ref{fact:EC-conservation} gives for every nonnegative $f$,
\[
    \Ent_{\widehat\nu}(f)
    \leq e^\gamma\widehat{\E}[\Ent_{\widehat\nu_1}(f)]\leq C_{\mathrm E}e^\gamma
          \widehat{\E}[\mathcal H_{\nu_1}^\partial(f)]
     \leq C_{\mathrm E}e^\gamma\mathcal H_\nu^\partial(f),
\]
which concludes the proof.
\end{proof}

\section{Absorbing the boundary form}\label{sec:boundary}\
This section proves Proposition~\ref{prop:defect-absorption}. Throughout this section, we fix $M\in \{0,\ldots, n-2\}$, and abbreviate $R=R_M\geq1$,
$\mu=\mu_{J,h,M}$, and write
$c_i(\sigma):=c_i(\sigma;\mu_{J,h})$ as the flip probability of $\sigma$. For $k\in\{-n,-n+2,\ldots,n\}$, define the magnetization slices
\[
\Omega_k:=\bigg\{\sigma\in\Omega:\sum_{i=1}^n\sigma_i=k\bigg\}.
\]
For all other $k$, it is understood that $\Omega_k=\emptyset$. Recall that the upward (resp. downward) magnetization drift $q_+$ (resp. $q_-$) are defined as
\[
q_+(\sigma)=P_{\mu_{J,h}}(\sigma,\Omega_{k+2})=\frac1n\sum_{i:\sigma_i=-1}
    \bigl(1-\phi(\ell_i(\sigma))\bigr),
\qquad
q_-(\sigma)=P_{\mu_{J,h}}(\sigma,\Omega_{k-2})=\frac1n\sum_{i:\sigma_i=+1}
    \phi(\ell_i(\sigma)).
\]
\begin{lemma}\label{prop:boundary-stability}
Suppose $\omega_M(J)<1/2$, and put $r_0:=\lfloor2R/3\rfloor$. Then,
\begin{enumerate}
\item[(a)] Under \eqref{assumption:stability},
\[
    q_-(\sigma)\le\frac R{3n},\qquad \sigma\in\Omega_{n-2r_0},
\]
and
\[
    q_-(\sigma)\le\frac45q_+(\sigma),\qquad \sigma\in\Omega_{n-2r},\quad r_0<r\le R.
\]
\item[(b)] Under \eqref{assumption:stability-entropy},
\[
    \frac1n\sum_{i:\sigma_i=+1}\psi(\ell_i(\sigma))\le\frac Rn,\qquad \sigma\in\Omega_{n-2r},\quad 1\le r\le R.
\]
\end{enumerate}
\end{lemma}
\begin{proof}
Fix $1\le r\le R$ and $\sigma\in\Omega_{n-2r}$. Define
\[
    T:=\{j:\sigma_j=-1\},\qquad v_i:=4\sum_{j\in T, j\ne i}J_{ij}.
\]
Then $|T|=r$ and $\ell_i(\sigma)=\ell_i(\1)-v_i$. Let $B:=\{i:\sigma_i=+1,\ v_i>1\}$. If $\sigma_i=+1$ and $i\notin B$, then $v_i\le1$, so
\[
    \phi(\ell_i(\1)-v_i)\le\phi(\ell_i(\1)-1)\le e^2\phi(\ell_i(\1)).
\]
If $i\in B$, then $0\le-\phi'\le1/2$ gives
\[
    \phi(\ell_i(\1)-v_i)-\phi(\ell_i(\1)-1)
    =\int_1^{v_i}-\phi'(\ell_i(\1)-s)\,ds
    \le\frac{v_i-1}{2}\le\frac{v_i^2}{8},
\]
where the last inequality follows from $(v_i-2)^2\ge0$. Combining both cases and applying \eqref{assumption:stability}, we obtain
\begin{equation}\label{ineq:sigma-stability}
    q_-(\sigma)\le\frac{e^2}{n}\sum_{i=1}^n c_i(\1)+\frac1{8n}\sum_{i\in B}v_i^2
    \le\frac R{4n}+\frac1{8n}\sum_{i\in B}v_i^2.
\end{equation}
It remains to estimate the last sum. First, we claim that $|B|<R$. Otherwise, choose $E\subseteq B$ with $|E|=R$ and set $S:=E\cup T$. Since $E$ and $T$ are disjoint, subtracting a scalar matrix from $J_{S,S}$ leaves the block $J_{E,T}$ unchanged. Thus, using $|S|=R+r\le2R$,
\[
    \|J_{E,T}\|_{\mathrm{op}}
    \le\inf_{t\in\R}\|J_{S,S}-tI\|_{\mathrm{op}}
    =\frac{\lambda_{\max}(J_{S,S})-\lambda_{\min}(J_{S,S})}{2}
    \le\frac{\omega_M(J)}2<\frac14.
\]
Since $v_i>1$ for every $i\in E$, this gives
\[
    R<\sum_{i\in E}v_i
    =4\1^\top J_{E,T}\1
    \le4\|J_{E,T}\|_{\mathrm{op}}\sqrt{Rr}<R,
\]
a contradiction. Hence $|B|<R$. Applying the same block estimate with $S=B\cup T$, we obtain
\[
    \sum_{i\in B}v_i^2
    =16\|J_{B,T}\1\|_2^2
    \le16\|J_{B,T}\|_{\mathrm{op}}^2r\le r.
\]
Substituting into \eqref{ineq:sigma-stability} yields
\[
    q_-(\sigma)\le\frac R{4n}+\frac r{8n}.
\]
Since $r_0\le2R/3$, this proves $q_-(\sigma)\le R/(3n)$ on $\Omega_{n-2r_0}$ when $r_0\ge1$. If $r_0=0$, the same conclusion follows directly from \eqref{assumption:stability}.

To control the resampling probabilities at all coordinates, set $S:=B\cup T$. Every index satisfying $v_i>1$ belongs to $S$. Set $A=J-\operatorname{diag}(J)$. Centering $J_{S,S}$ at its spectral midpoint and using that the norm of its diagonal part is at most its operator norm, we obtain
\[
    \|A_{S,S}\|_{\mathrm{op}}
    \le2\inf_{t\in\R}\|J_{S,S}-tI\|_{\mathrm{op}}
    \le\omega_M(J).
\]
The preceding estimate for $\phi$, now summed over all coordinates, therefore gives
\begin{equation}\label{eq:total-minus-rate}
\begin{split}
    \sum_{i=1}^n\phi(\ell_i(\sigma))
    &\le e^2\sum_{i=1}^n c_i(\1)+\frac18\sum_{i\in S}v_i^2\\
    &\le\frac R4+2\|A_{S,T}\1\|_2^2
    \le\frac R4+2\omega_M(J)^2r<\frac R4+\frac r2.
\end{split}
\end{equation}
For $r_0<r\le R$, we have $r>2R/3$, and hence
\[
    q_-(\sigma)\le\frac r{2n},\qquad
    q_+(\sigma)-q_-(\sigma)
    =\frac1n\left(r-\sum_{i=1}^n\phi(\ell_i(\sigma))\right)
    \ge\frac r{8n}.
\]
Consequently,
\[
    \frac{q_-(\sigma)}{q_+(\sigma)}
    \le\frac{r/(2n)}{r/(2n)+r/(8n)}=\frac45,
\]
which proves part (a). We now prove part (b). Differentiation gives
\[
    -\psi'(t)=2\bigl[\phi(t)^2+(1-\phi(t))(\psi(t)-\phi(t))\bigr].
\]
Since $\psi(t)\ge\phi(t)\log4\ge\phi(t)$ and $0<\phi(t)<1$, we obtain $0\le-\psi'(t)\le2\psi(t)$. Moreover, using $x\log(1/x)\le1/e$,
\[
    -\psi'(t)\le2+2(1-\phi(t))\psi(t)
    =2+2\phi(t)(1-\phi(t))\log\frac1{\phi(t)(1-\phi(t))}
    \le2+\frac2e<3.
\]
The bound $-\psi'\le2\psi$ implies $(e^{2t}\psi(t))'\ge0,$ and therefore $\psi(\ell_i(\1)-1)\le e^2\psi(\ell_i(\1))$. Thus, if $\sigma_i=+1$ and $i\notin B$, monotonicity gives
\[
    \psi(\ell_i(\1)-v_i)\le\psi(\ell_i(\1)-1)\le e^2\psi(\ell_i(\1)).
\]
For $i\in B$, the bound $-\psi'\le3$ gives
\[
    \psi(\ell_i(\1)-v_i)-\psi(\ell_i(\1)-1)
    =\int_1^{v_i}-\psi'(\ell_i(\1)-s)\,ds
    \le3(v_i-1)\le\frac34v_i^2.
\]
Summing and applying \eqref{assumption:stability-entropy}, we conclude that
\[
\begin{aligned}
    \sum_{i:\sigma_i=+1}\psi(\ell_i(\sigma))
    \le e^2\sum_{i=1}^n\psi(\ell_i(\1))+\frac34\sum_{i\in B}v_i^2
    \le e^2\delta_*R+\frac{3r}{4}
    \le R,
\end{aligned}
\]
which proves part (b).
\end{proof}

\noindent We now use Lemma~\ref{prop:boundary-stability} to control $\mathcal{Q}_M(f;a)$. Recall we denoted $\mu\equiv \mu_{J,h,M}$ throughout this section. For $k=M,M+2,\ldots,n$, denote
\[
    \widetilde{\mathcal Q}_k(f;a):=\sum_{\sigma\in\Omega_k}\mu(\sigma)q_-(\sigma)(f(\sigma)-a)^2,
    \qquad f:\widehat\Omega_M\to\R,\quad a\in\R.
\]
These are the quadratic boundary forms from Section~\ref{sec:proof-technique} with the normalization fixed at $\widehat\Omega_M$:
\[
    \widetilde{\mathcal Q}_k(f;a)=\mu(\widehat\Omega_k)\mathcal Q_k(f;a),
    \qquad \widetilde{\mathcal Q}_M(f;a)=\mathcal Q_M(f;a).
\]
For $k=M,M+2,\ldots,n-2$, write the contribution to $\mathcal E_\mu(f)$ between $\Omega_k$ and $\Omega_{k+2}$ in terms of upward transitions as
\[
    \mathcal E_\mu^k(f):=\frac1n\sum_{\sigma\in\Omega_k}\sum_{i:\sigma_i=-1}\mu(\sigma)c_i(\sigma)\bigl(f(\sigma)-f(\sigma^{(i)})\bigr)^2.
\]
By reversibility, we have
\[
\mathcal E_\mu(f)=\sum_{k=M,M+2,\ldots,n-2}\mathcal E_\mu^k(f).
\]
\begin{lemma}\label{lemma:boundary-transport}
Fix integers $K \in \{M, M+2, \dots, n-2\}$ and $1\le L\le (n-K)/2$. Suppose that, for some $\alpha\in(0,1)$
\begin{equation}\label{eq:boundary-transport-drift}
    q_-(\sigma)\le\alpha q_+(\sigma),
    \qquad \sigma\in\Omega_k,\quad k=K,K+2,\ldots,K+2L-2.
\end{equation}
Choose $\theta>0$ such that $\rho:=\alpha(1+\theta)<1$. Then, for every $f:\widehat\Omega_M\to\R$, $a\in\R$, 
\begin{equation}\label{ineq:boundary-transport-step}
    \widetilde{\mathcal Q}_K(f;a)
     \leq \rho^L \, \widetilde{\mathcal Q}_{K+2L}(f;a)
       +\frac{\rho}{\theta} \, \mathcal E_\mu(f) 
\end{equation}
\end{lemma}
\begin{proof}
Set $g:=f-a$. For each $k = K, K+2,\dots, K+2L-2$, the drift assumption \eqref{eq:boundary-transport-drift} and the definition of $q_+$ give
\[
    \widetilde{\mathcal Q}_k(f;a)
    \le\alpha\sum_{\sigma\in\Omega_k}\mu(\sigma)q_+(\sigma)g(\sigma)^2
    =\frac{\alpha}{n}\sum_{\sigma\in\Omega_k}\sum_{i:\sigma_i=-1}\mu(\sigma)c_i(\sigma)g(\sigma)^2.
\]
Using $(a+b)^2\leq (1+\theta) a^2+(1+1/\theta)b^2$, we have for each $i$ with $\sigma_i=-1$,
\[
    g(\sigma)^2
    \le(1+\theta)g(\sigma^{(i)})^2
       +\left(1+\frac1\theta\right)\bigl(g(\sigma)-g(\sigma^{(i)})\bigr)^2.
\]
Moreover, we have $\mu(\sigma)c_i(\sigma)=\mu(\sigma^{(i)})c_i(\sigma^{(i)})$ by reversibility, thus 
\[
    \frac1n\sum_{\sigma\in\Omega_k}\sum_{i:\sigma_i=-1}\mu(\sigma)c_i(\sigma)g(\sigma^{(i)})^2
    =\sum_{\tau\in\Omega_{k+2}}\mu(\tau)q_-(\tau)g(\tau)^2
    =\widetilde{\mathcal Q}_{k+2}(f;a).
\]
Combining these estimates with the definition of $\mathcal E_\mu^k$, we obtain
\[
     \widetilde{\mathcal Q}_k(f;a)
    \leq \rho\,\widetilde{\mathcal Q}_{k+2}(f;a)
       +\frac{\rho}{\theta}\mathcal E_\mu^k(f).
\]
Iterating from $K$ to $K+2L-2$ gives
\[
\begin{aligned}
    \widetilde{\mathcal Q}_K(f;a)
    \leq \rho^L\widetilde{\mathcal Q}_{K+2L}(f;a)
       +\frac{\rho}{\theta}
          \sum_{j=0}^{L-1}\rho^j\mathcal E_\mu^{K+2j}(f)
    \leq \rho^L\widetilde{\mathcal Q}_{K+2L}(f;a)
       +\frac{\rho}{\theta}\mathcal E_\mu(f),
\end{aligned}
\]
where the last inequality follows from $\rho \in (0,1)$ and $\mathcal E_\mu(f)
=\sum_{k=M,M+2,\ldots,n-2}\mathcal E_\mu^k(f)$.
\end{proof}

We remark that the preceding proof resembles the classical Dirichlet comparison method~\cite{DSC93}. Indeed, adjoining a vertex $*$
with $f(*)=a$, we may view $\widetilde{\mathcal Q}_k(f;a)$ and
$\mathcal E_\mu^k(f)+\widetilde{\mathcal Q}_{k+2}(f;a)$ as
Dirichlet forms on the same augmented space. Then, the comparison
replaces each boundary edge by two-step paths, $(\sigma,*)\mapsto(\sigma,\sigma^{(i)},*),\sigma_i=-1.$ 

We now turn to the entropy estimates. In the lemma below, we bound the entropic boundary term by quadratic terms, allowing us to subsequently apply Lemma~\ref{lemma:boundary-transport}. 

\begin{lemma}\label{lemma:entropic-boundary-reduction}
Assume $\omega_M(J)<1/2$ and \eqref{assumption:stability-entropy}. Then, for every $f:\widehat\Omega_M\to[0,\infty)$ and $a\ge0$,
\begin{equation}\label{eq:entropic-boundary-reduction}
    \mathcal B_M(f^2;a^2)\lesssim \widetilde{\mathcal Q}_{M+2}(f;a)+\mathcal E_\mu(f).
\end{equation}
\end{lemma}
\begin{proof}
Since \eqref{assumption:stability-entropy} implies \eqref{assumption:stability}, the estimate \eqref{eq:total-minus-rate} in Lemma~\ref{prop:boundary-stability}, applied with $r=R$, gives
\[
    q_+(\sigma)=\frac1n\left(R-\sum_{i:\sigma_i=-1}\phi(\ell_i(\sigma))\right)\ge\frac1n\left(R-\sum_{i=1}^n\phi(\ell_i(\sigma))\right)\ge\frac R{4n},\qquad \sigma\in\Omega_M.
\]
We next establish the two-point estimate
\begin{equation}\label{eq:claim:log-sobolev}
    \ent_p(x^2,y^2)
    \leq 2p(1-p)\log\bigg(\frac1{p(1-p)}\bigg)(x-y)^2,
    \qquad p\in(0,1),\quad x,y\in\R.
\end{equation}
Indeed, the two-point log-Sobolev
inequality~\cite[Appendix~A, Theorem~A.2]{DSC96} gives 
\[
    \ent_p(x^2,y^2)
    \leq p(1-p)\frac{\log\left[p/(1-p)\right]}{2p-1}(|x|-|y|)^2,
\]
where the quotient is interpreted by continuity at $p=1/2$. Note that
\[
\frac{d}{dp}\bigg[(2p-1)\log\bigg(\frac1{p(1-p)}\bigg)
-\frac12\log\bigg(\frac p{1-p}\bigg)\bigg]=2\log\bigg(\frac{1}{p(1-p)}\bigg)+\frac{1}{2p(1-p)}-4>0,
\]
where the last inequality holds since $p(1-p)\leq 1/4$ and $2\log 4-2>0$. The expression in the brackets vanishes for $p=1/2$, thus for $p\geq 1/2$, we have $\frac{\log\left[p/(1-p)\right]}{2p-1}\leq 2\log(\frac{1}{p(1-p)})$. Together with $(|x|-|y|)^2\leq(x-y)^2$, this proves
\eqref{eq:claim:log-sobolev}.

For $\sigma_i=+1$, we have $c_i(\sigma)=\phi(\ell_i(\sigma))$. Thus, by the definition of $\mathcal B_M$, \eqref{eq:claim:log-sobolev}, and Lemma~\ref{prop:boundary-stability}-(b),
\begin{align*}
    \mathcal B_M(f^2;a^2)&=\frac1n\sum_{\sigma\in\Omega_M}\sum_{i:\sigma_i=+1}
        \frac{\mu(\sigma)}{1-c_i(\sigma)}
        \ent_{1-c_i(\sigma)}(f(\sigma)^2,a^2)\\
    &\leq\frac2n\sum_{\sigma\in\Omega_M}
        \mu(\sigma)(f(\sigma)-a)^2
        \sum_{i:\sigma_i=+1}\psi(\ell_i(\sigma))\\
    &\leq\frac{2R}{n}\sum_{\sigma\in\Omega_M}
        \mu(\sigma)(f(\sigma)-a)^2
    \leq8\sum_{\sigma\in\Omega_M}
        \mu(\sigma)q_+(\sigma)(f(\sigma)-a)^2.
\end{align*}
Finally, using $(f(\sigma)-a)^2
\leq2(f(\sigma^{(i)})-a)^2
   +2(f(\sigma)-f(\sigma^{(i)}))^2$ and reversibility give
\[
    \sum_{\sigma\in\Omega_M}\mu(\sigma)q_+(\sigma)(f(\sigma)-a)^2\le2\widetilde{\mathcal Q}_{M+2}(f;a)+2\mathcal E_\mu^M(f).
\]
Therefore,
\[
    \mathcal B_M(f^2;a^2)\le16\Big(\widetilde{\mathcal Q}_{M+2}(f;a)+\mathcal E_\mu^M(f)\Big)\le16\Big(\widetilde{\mathcal Q}_{M+2}(f;a)+\mathcal E_\mu(f)\Big).
\]
which proves the lemma. 
\end{proof}

We prove Proposition~\ref{prop:defect-absorption} by using
Lemma~\ref{lemma:boundary-transport}. For entropy,
Lemma~\ref{lemma:entropic-boundary-reduction} reduces the proof
to the same argument applied to $\sqrt{f}$.

\begin{proof}[Proof of Proposition~\ref{prop:defect-absorption}]
We first prove the boundary estimates.  Set $r_0:=\lfloor2R/3\rfloor$ and $L:=R-r_0=\lceil R/3\rceil$. Assume \eqref{assumption:stability} and fix $f:\widehat\Omega_M\to\R$. By Lemma~\ref{prop:boundary-stability}, $q_-(\sigma)\le R/(3n)$ for $\sigma\in\Omega_{n-2r_0}$. Hence, for every $a\in\R$,
\begin{equation}\label{eq:inner-layer-defect}
    \widetilde{\mathcal Q}_{M+2L}(f;a) = \widetilde{\mathcal Q}_{n-2r_0}(f;a)
    =\sum_{\sigma\in\Omega_{n-2r_0}}\mu(\sigma)q_-(\sigma)(f(\sigma)-a)^2
    \le\frac R{3n}\EE_\mu[(f-a)^2].
\end{equation}
The same lemma gives $q_-\le(4/5)q_+$ on $\Omega_{n-2r}$ whenever $r_0<r\le R$. Thus, applying Lemma~\ref{lemma:boundary-transport} with $\alpha = 4/5$, $K = M$, and using the terminal estimate \eqref{eq:inner-layer-defect}, we obtain, for some universal $\rho \in (0,1)$,
\[
  \mathcal{Q}_M(f;a) =  \widetilde{\mathcal Q}_M(f;a) \lesssim \mathcal{E}_{\mu}(f) + \rho^{L} \, \widetilde{\mathcal Q}_{n-2r_0}(f;a)
    \lesssim \mathcal E_\mu(f)+\frac R{n}\rho^{L}\EE_\mu[(f-a)^2],
\]
Hence, for $a=\EE_\mu[f]$,
\[
    \mathcal Q_M(f)\lesssim\mathcal E_\mu(f)+\frac Rn\rho^L\Var_\mu(f).
\]
We next prove the entropic boundary estimate. Assume \eqref{assumption:stability-entropy}, which implies \eqref{assumption:stability}. For $f\ge0$ and $a\ge0$, Lemma~\ref{lemma:entropic-boundary-reduction} and Lemma~\ref{lemma:boundary-transport} give
\begin{equation}\label{eq:entropic-boundary-moment}
\begin{aligned}
    \mathcal B_M(f;a^2)
    \lesssim\widetilde{\mathcal Q}_{M+2}(\sqrt f;a)+\mathcal E_\mu(\sqrt f)
    \lesssim\mathcal E_\mu(\sqrt f)+\frac Rn\rho^{L-1}\EE_\mu[(\sqrt f-a)^2].
\end{aligned}
\end{equation}
This also covers $L=1$, since then $M+2=n-2r_0$. Applying $\Var(\sqrt f)\le\Ent(f)$ to each conditional distribution gives
\[
    \mathcal E_\mu(\sqrt f)
    =\frac1n\sum_{i=1}^n\EE_\mu\!\left[\Var_{\mu(\cdot\mid\sigma_{-i})}(\sqrt f)\right]
    \le\frac1n\sum_{i=1}^n\EE_\mu\!\left[\Ent_{\mu(\cdot\mid\sigma_{-i})}(f)\right]
    =\mathcal H_\mu(f).
\]
Taking $a=\EE_\mu[\sqrt f]$ in \eqref{eq:entropic-boundary-moment} and using $\Var_\mu(\sqrt f)\le\Ent_\mu(f)$, we obtain
\[
    \mathcal B_M(f)\lesssim\mathcal H_\mu(f)+\frac Rn\rho^{L-1}\Ent_\mu(f).
\]
Since $L\ge R/3$, there is a universal $c>0$ such that
\[
    R\rho^L\le R\rho^{L-1}\le\rho^{-1}R\rho^{R/3}\lesssim e^{-cR}.
\]
Absorbing the implicit constants into $C$, proves the two boundary estimates.

\medskip

It remains to prove the second inequality in each part of the proposition. We first prove both comparison estimates when $J=0$, so that $\mu=\mu_{0,h,M}$. Under \eqref{assumption:stability}, the ambient product measure satisfies
\[
    c_i(\sigma)=c_i(\1)\quad\text{if }\sigma_i=+1,
    \qquad c_i(\sigma)=1-c_i(\1)\quad\text{if }\sigma_i=-1.
\]
Consequently, for $\sigma\in\Omega_{n-2r}$,
\[
    q_-(\sigma)\le\frac{\delta_*R}{n}\le\frac R{3n},
    \qquad q_+(\sigma)\ge\frac{r-\delta_*R}{n}.
\]
For $r_0<r\le R$, we therefore have
\[
    \frac{q_-(\sigma)}{q_+(\sigma)}
    \le\frac{\delta_*}{2/3-\delta_*}<\frac14.
\]
Fix $f:\widehat\Omega_M\to\R$ and put $a:=\EE_\mu[f]$. Applying Lemma~\ref{lemma:general-boundary-decomp} and Lemma~\ref{lemma:boundary-transport}, 
\[
\begin{split}
    \Var_\mu(f)
    &\leq n\bigl(\mathcal E_\mu(f)+\mathcal Q_M(f;a)\bigr)\\
    &\leq \frac32n\mathcal E_\mu(f)
        +n2^{-L}\widetilde{\mathcal Q}_{n-2r_0}(f;a)
    \leq \frac32n\mathcal E_\mu(f)
        +\frac R3\,2^{-L}\Var_\mu(f).
\end{split}
\]
where the last inequality follows from \eqref{eq:inner-layer-defect}. Since $R\le3L$ and $L\le2^{L-1}$,
\[
    \frac R3\,2^{-L}\le L2^{-L}\le\frac12.
\]
Absorbing the variance term yields
\begin{equation}\label{eq:product-poincare}
    \Var_\mu(f)\lesssim n\mathcal E_\mu(f).
\end{equation}
Now assume \eqref{assumption:stability-entropy}, fix $f:\widehat\Omega_M\to[0,\infty)$, and put $a:=\EE_\mu[\sqrt f]$. Combining \eqref{eq:entropic-boundary-moment} with \eqref{eq:product-poincare}, we obtain
\[
    \mathcal B_M(f;a^2)
    \lesssim\mathcal E_\mu(\sqrt f)+\frac Rn\rho^{L-1}\Var_\mu(\sqrt f)
    \lesssim\mathcal E_\mu(\sqrt f)\le\mathcal H_\mu(f),
\]
where we used $R\rho^{L-1}\lesssim1$. Lemma~\ref{lemma:general-boundary-decomp} therefore gives
\[
    \Ent_\mu(f)\le n\bigl(\mathcal H_\mu(f)+\mathcal B_M(f;a^2)\bigr)
    \lesssim n\mathcal H_\mu(f).
\]
Set $A: = J - \diag(J)$. For general $J$, we compare $\mu$ with
\[
    \nu:=\mu_{0,h+2A\1,M}
\]
by the bounded perturbation argument of Holley and Stroock~\cite{HS87}, which we give explicitly below. The ambient product measure satisfies
\begin{equation}\label{eq:shifted-product-all-plus}
    \mu_{0,h+2A\1}(\sigma_i=-1)=\phi(\ell_i(\1))=c_i(\1).
\end{equation}
Thus, \eqref{assumption:stability} and \eqref{assumption:stability-entropy} imply the corresponding assumptions for this product measure. The product estimates just proved give, under the respective assumptions,
\[
    \Var_\nu(f)\lesssim n\mathcal E_\nu(f),
    \qquad \Ent_\nu(f)\lesssim n\mathcal H_\nu(f).
\]
For $\sigma\in\widehat\Omega_M$, put $T:=\{i:\sigma_i=-1\}$ and let $\one_T$ denote its indicator vector. Since $\sigma=\1-2\one_T$,
\[
    \langle A\sigma,\sigma\rangle+\langle h,\sigma\rangle
    =\langle h+2A\1,\sigma\rangle+4\one_T^\top A\one_T-\langle A\1,\1\rangle.
\]
It follows that
\begin{equation}\label{eq:capped-product-density}
    \Psi(\sigma):=\frac{d\mu}{d\nu}(\sigma)
    \propto\exp\!\left(4\one_T^\top A\one_T\right).
\end{equation}
For every nonempty $T$ with $|T|\le R$, centering $J_{T,T}$ at its spectral midpoint gives
\[
    \|A_{T,T}\|_{\mathrm{op}}
    \le2\inf_{t\in\R}\|J_{T,T}-tI\|_{\mathrm{op}}
    \le\omega_M(J).
\]
Hence $|\one_T^\top A\one_T|\le |T|\omega_M(J)\le R\omega_M(J)$, with the same bound holding trivially for $T=\varnothing$. By \eqref{eq:capped-product-density},
\begin{equation}\label{eq:capped-product-oscillation}
    \operatorname{osc}_{\widehat\Omega_M}\log\Psi\le8R\omega_M(J).
\end{equation}
Set $\Psi_{\min}:=\min_{\widehat\Omega_M}\Psi$ and $\Psi_{\max}:=\max_{\widehat\Omega_M}\Psi$. Then
\[
    \Var_\mu(f)=\inf_{a\in\R}\EE_\nu[\Psi(f-a)^2]
    \le\Psi_{\max}\Var_\nu(f).
\]
Similarly, since $f\log(f/b)-f+b\ge0$ for $b>0$, the variational formula for entropy gives
\[
    \Ent_\mu(f)
    =\inf_{b>0}\EE_\nu\!\left[\Psi\left(f\log\frac fb-f+b\right)\right]
    \le\Psi_{\max}\Ent_\nu(f).
\]
To compare the Dirichlet and conditional entropy forms, fix a coordinate $i$ and values $\eta$ for the remaining spins such that $F:=\{\sigma\in\widehat\Omega_M:\sigma_{-i}=\eta\}$ is non-empty. Then,
\[
\begin{aligned}
    \mu(F)\Var_{\mu(\cdot\mid F)}(f)
    =\inf_{a\in\R}\sum_{\sigma\in F}\mu(\sigma)(f(\sigma)-a)^2
    \ge\Psi_{\min}\nu(F)\Var_{\nu(\cdot\mid F)}(f).
\end{aligned}
\]
The same argument for entropy gives
\[
\begin{aligned}
    \mu(F)\Ent_{\mu(\cdot\mid F)}(f)
    =\inf_{b>0}\sum_{\sigma\in F}\mu(\sigma)
      \left(f(\sigma)\log\frac{f(\sigma)}b-f(\sigma)+b\right)
    \ge\Psi_{\min}\nu(F)\Ent_{\nu(\cdot\mid F)}(f).
\end{aligned}
\]
Summing over all choices of $i$ and $\eta$, and dividing by $n$, we obtain
\[
    \mathcal E_\mu(f)\ge\Psi_{\min}\mathcal E_\nu(f),
    \qquad \mathcal H_\mu(f)\ge\Psi_{\min}\mathcal H_\nu(f).
\]
Combining these inequalities with the product estimates and \eqref{eq:capped-product-oscillation} proves, under the respective assumptions, we have
\[
    \Var_\mu(f)\lesssim n e^{8R\omega_M(J)}\mathcal E_\mu(f),
    \qquad \Ent_\mu(f)\lesssim n e^{8R\omega_M(J)}\mathcal H_\mu(f),
\]
which concludes the proof.
\end{proof}

\section{Lifting the core functional inequalities}\label{sec:lifting}

This section proves Theorem~\ref{theorem:lifting} from Theorem~\ref{theorem:core-func-ineq}. To extend Poincar\'e and MLSI inequalities beyond $\capOmega_M$, we need a mechanism to control excursions outside the cap. For this purpose, we employ a Foster-Lyapunov condition (see e.g.~\cite{MT09}):  let $(X_t)_{t \in \Z_{\geq 0}}$ be a discrete-time Markov chain with kernel $P$. A Foster-Lyapunov condition towards a set $C$ is the existence of a function $V \geq 1$ such that
\[
    \EE[V(X_{t+1})-V(X_t)\mid X_t=x]
    \leq-\lambda V(x)+b\,\one_C(x),
    \qquad 0<\lambda<1,\quad b<\infty.
\]
or, equivalently, $PV \leq (1-\lambda)V +b\, \one_C$. Writing $\tau_C := \inf \{ t \in \Z_{\geq 0} : X_t \in C\}$, this condition gives
\[
    \EE_x[(1-\lambda)^{-\tau_C}] \leq V(x), \qquad \PP_x(\tau_C > t) \leq V(x)(1-\lambda)^t, \qquad t \in \N
\]
see \cite[Proposition 4.1]{Bax05} and \cite[Theorem 15.2.5]{MT09}. Thus, a Foster-Lyapunov condition bounds the duration of excursions outside $C$, with $\lambda$ determining the decay rate. This drift condition can be combined with a Poincar\'e inequality
on $C$ to obtain a Poincar\'e inequality on the full state space;
see \cite{BBCG08,BCG08,CG14} for the continuous setting.
For a related discrete-time argument based on drift and
minorization, see \cite[Theorem~1]{TM22}. 

For Glauber dynamics, the unnormalized generator satisfies
\[
    \mathcal L:=n(P_{J,h}-I),
    \qquad
    -\EE_{\mu_{J,h}}[f\mathcal Lf]
    =n\mathcal E_{\mu_{J,h}}(f),
\]
where $\mathcal{L}f(\sigma) = \sum_{\tau \neq \sigma}\mathcal L(\sigma, \tau)(f(\tau)-f(\sigma))$. In this section, denote Hamming distance for $\sigma, \tau \in \Omega$ as
\[
d(\sigma,\tau) :=\bigl|\{i\in[n]:\sigma_i\neq\tau_i\}\bigr|
\]
and write $R= R_M$. For nonempty $B\subseteq\Omega$, write
$d(\sigma,B):=\min_{\tau\in B}d(\sigma,\tau)$.

We use the following Lyapunov inequality from Taghvaei and Mehta~\cite[equation~(7), proof of Theorem~1]{TM22}, whose continuous-time form goes back to Cattiaux, Guillin, Wang, and Wu~\cite[Lemma~2.12]{CGWW09}. We include a short proof for completeness.

\begin{lemma}[\cite{CGWW09,TM22}]\label{lemma:ground-state-inequality}
For every $W:\Omega \to(0,\infty)$ and $f:\Omega\to\R$,
\[
    \EE_{\mu_{J,h}}\left[-\frac{\mathcal LW}{W}f^2\right]
    \leq n\mathcal E_{\mu_{J,h}}(f).
\]
In particular, the result holds for the Glauber dynamics on $\capOmega_S$, with invariant measure $\mu_{J,h,S}$, generator $\mathcal{L} = n(P_{J,h,S} - I)$, and $W, f$ defined on $\capOmega_S$.  
\end{lemma}

\begin{proof}
Using reversibility,
\[
    \EE_{\mu_{J,h}}\left[-\frac{\mathcal LW}{W}f^2\right]
    =\frac12\sum_{\sigma\neq\tau}
        \mu_{J,h}(\sigma)\mathcal L(\sigma,\tau)
        \bigl(W(\sigma)-W(\tau)\bigr)
        \left(
            \frac{f(\sigma)^2}{W(\sigma)}
            -\frac{f(\tau)^2}{W(\tau)}
        \right).
\]
Now note that we have the following identity:
\[
    \bigl(f(\sigma)-f(\tau)\bigr)^2-(W(\sigma)-W(\tau))\left(
        \frac{f(\sigma)^2}{W(\sigma)}
        -\frac{f(\tau)^2}{W(\tau)}
    \right) =
    W(\sigma)W(\tau)
    \left(
        \frac{f(\sigma)}{W(\sigma)}
        -\frac{f(\tau)}{W(\tau)}
    \right)^2
\]
Therefore, 
\[
    \begin{aligned}
        n\mathcal E_{\mu_{J,h}}(f)
        -\EE_{\mu_{J,h}}\left[-\frac{\mathcal LW}{W}f^2\right]
        =
        \frac12\sum_{\sigma\neq\tau}
            \mu_{J,h}(\sigma)\mathcal L(\sigma,\tau)
            W(\sigma)W(\tau)
            \left(
                \frac{f(\sigma)}{W(\sigma)}
                -\frac{f(\tau)}{W(\tau)}
            \right)^2
        \geq0.
    \end{aligned}
\]
The same proof applies to the Glauber dynamics with $\mu_{J,h}$ replaced by $\mu_{J,h,S}$, $\mathcal L=n(P_{J,h,S}-I)$, and $W,f$ defined on $\capOmega_S$.
\end{proof}

\subsection{Poincar\'e lifting}

We first prove Theorem~\ref{theorem:lifting}-(a). We follow the centering argument in the proof
of \cite[Theorem~1]{TM22}, and use Theorem~\ref{theorem:core-func-ineq} for the estimate on the cap. As in \cite[Theorem~1]{TM22},  the main idea is to split an upper bound for the variance into its contributions on the cap $\capOmega_M$ and outside the cap $\widehat\Omega_S \setminus \capOmega_M$, and then estimate each contribution. Write $p:=\mu_{J,h,S}(\capOmega_M)$.
For every $f:\capOmega_S\to\R$, 
\begin{equation}\label{variance-bound-decomposition}
    \begin{aligned}
        \Var_{\mu_{J,h,S}}(f)
        &\leq
        \EE_{\mu_{J,h,S}}
        \left[(f-\EE_{\mu_{J,h,M}}[f])^2\right]\\
        &=p\Var_{\mu_{J,h,M}}(f)
        +\EE_{\mu_{J,h,S}}
        \left[
            (f-\EE_{\mu_{J,h,M}}[f])^2
            \one_{\capOmega_S\setminus\capOmega_M}
        \right].
    \end{aligned}
\end{equation}
The first term is controlled by Theorem \ref{theorem:core-func-ineq} (a). The following lemma controls the second term. 

\begin{lemma}\label{lemma:variance-tail-control}
Under the assumptions of Theorem~\ref{theorem:lifting}(a),
every $f:\capOmega_S\to\R$ satisfies
\begin{equation}\label{eq:lyapunov-tail}
    \EE_{\mu_{J,h,S}}
    \left[
        (f-\EE_{\mu_{J,h,M}}[f])^2
        \one_{\capOmega_S\setminus\capOmega_M}
    \right]
    \leq\frac n\zeta\mathcal E_{\mu_{J,h,S}}(f)
        +\frac p\zeta\Var_{\mu_{J,h,M}}(f).
\end{equation}
\end{lemma}
\begin{proof}
Define
\[
    \mathcal L:=n(P_{J,h,S}-I),
    \qquad
    V(\sigma):=\max\{R,d(\sigma,\1)\}.
\]
Assumption~\eqref{assumption:stability} and parity give $R\geq1$, so that $V > 0$. For $\sigma\in\Omega_k$, we have $\mathcal LV(\sigma)=0$
when $k>M$ and $\mathcal LV(\sigma)=nq_-(\sigma)$ when $k=M$.
For $S\leq k<M$, \eqref{assumption:tail-stability} gives
\[
    \begin{aligned}
        \mathcal LV(\sigma)
        =n\left(
            P_{J,h}(\sigma,\Omega_{k-2})\one_{\{k>S\}}
            -P_{J,h}(\sigma,\Omega_{k+2})
        \right)
        \leq-\zeta\frac{n-k}{2}
        =-\zeta V(\sigma).
    \end{aligned}
\]
On $\Omega_M$, Lemma~\ref{prop:boundary-stability}(a) gives
$q_-\leq(4/5)q_+\leq R/n$. Hence
\[
    -\frac{\mathcal LV}{V}
    \geq
    \zeta\one_{\capOmega_S\setminus\capOmega_M}
    -\one_{\capOmega_M}.
\]
Let $g:=f-\EE_{\mu_{J,h,M}}[f]$.
By Lemma~\ref{lemma:ground-state-inequality},
\[
    \begin{aligned}
        n\mathcal E_{\mu_{J,h,S}}(f)
        =n\mathcal E_{\mu_{J,h,S}}(g)
        \geq
        \EE_{\mu_{J,h,S}}
        \left[-\frac{\mathcal LV}{V}g^2\right]
        \geq
        \zeta\EE_{\mu_{J,h,S}}
        \left[g^2\one_{\capOmega_S\setminus\capOmega_M}\right]
        -p\Var_{\mu_{J,h,M}}(f).
    \end{aligned}
\]
Rearranging proves the lemma. 
\end{proof}

\begin{proof}[Proof of Theorem~\ref{theorem:lifting}-(a)]
Combining \eqref{variance-bound-decomposition},
Lemma~\ref{lemma:variance-tail-control}, and
Theorem~\ref{theorem:core-func-ineq}(a), we obtain
\[
    \begin{aligned}
        \Var_{\mu_{J,h,S}}(f)
        \leq \frac n\zeta\mathcal E_{\mu_{J,h,S}}(f)
            +(1+\zeta^{-1})p\Var_{\mu_{J,h,M}}(f)
        \leq \frac{Cn}{\zeta}
            \exp\!\left(\frac{C\|J\|_{\op}}{\eps}\right)
            \mathcal E_{\mu_{J,h,S}}(f),
    \end{aligned}
\]
where we used $p\mathcal E_{\mu_{J,h,M}}(f)
\leq\mathcal E_{\mu_{J,h,S}}(f)$. This proves part (a) of the theorem. 
\end{proof}

\subsection{Entropy lifting}

We now turn to Theorem \ref{theorem:lifting}-(b) and lift the MLSI from $\widehat\Omega_M$ to $\widehat\Omega_S$. For $f \geq 0$, set $a : = (\E_{\mu_{J,h,M}} f)^{1/2}$. The choice $a^2=\EE_{\mu_{J,h,M}}[f]$ preserves the mean
when $f$ is replaced by $a^2$ on the cap. We again separate the contribution inside $\capOmega_M$ from the contribution on $\capOmega_S \setminus \capOmega_M$ and obtain
\begin{equation}\label{entropy-lifting-decomposition}
 \Ent_{\mu_{J,h,S}}(f) \lesssim_{\eta} p \Ent_{\mu_{J,h,M}}(f) + n \,\EE_{\mu_{J,h,S}}[(\sqrt{f} - a)^2 \one_{\capOmega_{S} \setminus \capOmega_M}]
\end{equation}
The above inequality is derived in the proof of Theorem \ref{theorem:lifting}-(b) using the entropy chain rule and Rothaus' inequality, see below. As in the Poincar\'e case, the first term is controlled by Theorem~\ref{theorem:core-func-ineq}-(b), whereas the second term is controlled using a Lyapunov drift argument. Note, however, that the second term has an additional factor of $n$, which requires an exponential Lyapunov function. The following lemma, which is the entropic equivalent of Lemma \ref{lemma:variance-tail-control}, controls the exterior quadratic form in (\ref{entropy-lifting-decomposition}) and is proved below.

\begin{lemma}\label{lemma:basin-entropy}
Suppose that the assumptions of Theorem~\ref{theorem:lifting}-(a)
hold and that $M\leq(1-\vartheta)n$ for some $\vartheta\in(0,1)$.
Then every $f:\capOmega_S\to[0,\infty)$ satisfies
\begin{equation}\label{eq:entropy-transfer}
    n\EE_{\mu_{J,h,S}}\left[
        (\sqrt f-a)^2
        \one_{\capOmega_S\setminus\capOmega_M}
    \right]
    \lesssim_{\zeta,\vartheta}
        n\mathcal E_{\mu_{J,h,S}}(\sqrt f)
        +p\Ent_{\mu_{J,h,M}}(f),
    \quad
    a:=(\EE_{\mu_{J,h,M}}[f])^{1/2}.
\end{equation}
\end{lemma}
\begin{proof}
Define
\[
    \mathcal L:=n(P_{J,h,S}-I) \quad \text{and} \quad V(\sigma)=\exp\left(\frac{\zeta}{2} \, d(\sigma,\widehat\Omega_M)\right)
\]
Take $\sigma\in\Omega_k$. If $k>M$, then $\mathcal LV(\sigma)=0$.
If $k=M$, then only an outward move changes $V$, and $S\leq M-2$, so
\[
    \mathcal LV(\sigma)
    =n(e^{\zeta/2}-1)P_{J,h,S}(\sigma,\Omega_{M-2})
    =n(e^{\zeta/2}-1)q_-(\sigma).
\]
Finally, when $S\leq k<M$, put $r:=(n-k)/2$. Then
\[
    \mathcal LV(\sigma)
    =nV(\sigma)\left[
        (e^{-\zeta/2}-1)P_{J,h,S}(\sigma,\Omega_{k+2})
        +(e^{\zeta/2}-1)P_{J,h,S}(\sigma,\Omega_{k-2})
    \right].
\]
Using \eqref{assumption:tail-stability} and
$nP_{J,h,S}(\sigma,\Omega_{k+2})\leq r$,
\[
\begin{aligned}
    -\frac{\mathcal LV}{V}(\sigma) \geq(e^{\zeta/2}-1)\left(\zeta r-\frac{\zeta}{2}r\right)
    \geq\frac{\zeta^2}{4}r,
\end{aligned}
\]
where we used $1-e^{-x}\leq x$ and $e^x-1\geq x$ for $x\geq0$. This also holds on the reflecting boundary $k=S$, where the outward move probability is zero. Therefore,
\begin{equation}\label{ineq:entropy-drift-bound-lemma}
    -\frac{\mathcal LV}{V}
    \geq\frac{\zeta^2R}{4}
        \one_{\widehat\Omega_S\setminus\widehat\Omega_M}
    -n(e^{\zeta/2}-1)q_-\one_{\Omega_M}.
\end{equation}
Let $f : \capOmega_S \rightarrow [0,\infty)$ and set $a:=(\E_{\mu_{J,h,M}}[f])^{1/2}$. Applying Lemma~\ref{lemma:ground-state-inequality}, we have
\[
    n\mathcal E_{\mu_{J,h,S}}(\sqrt f)
    =n\mathcal E_{\mu_{J,h,S}}(\sqrt f-a)
    \geq
    \EE_{\mu_{J,h,S}}\left[
        -\frac{\mathcal LV}{V}(\sqrt f-a)^2
    \right]
\]
Using drift bound (\ref{ineq:entropy-drift-bound-lemma}) and rearranging, 
\begin{equation}\label{ineq:entropy-drift-lemma-final}
      R\,\E_{\mu_{J,h,S}}\left[
        (\sqrt f-a)^2
        \one_{\widehat\Omega_S\setminus\widehat\Omega_M}
    \right] \lesssim_{\zeta}   n\mathcal E_{\mu_{J,h,S}}(\sqrt f) +
     n\E_{\mu_{J,h,S}}\left[
            q_-(\sqrt f-a)^2\one_{\Omega_M}
        \right].
\end{equation}
It remains to control the last term. From the definitions,
\[
        \E_{\mu_{J,h,S}}\left[
        q_-(\sqrt f-a)^2\one_{\Omega_M}
    \right]
    =p\,
        \mathcal Q_M(\sqrt f;a).
\]
Let $r_0:=\lfloor2R/3\rfloor$. By Lemma~\ref{prop:boundary-stability}-(a),
$q_-\leq R/(3n)$ on $\Omega_{n-2r_0}$ and $q_-/q_+\leq4/5$ on
$\Omega_k$ for $M\leq k<n-2r_0$. Thus, applying
Lemma~\ref{lemma:boundary-transport}, we obtain, for some
universal $\rho\in(0,1)$,
\[
\begin{aligned}
\mathcal Q_M(\sqrt f;a)
&\lesssim \mathcal E_{\mu_{J,h,M}}(\sqrt f)
+\rho^{R-r_0}\widetilde{\mathcal Q}_{n-2r_0}(\sqrt f;a)\\
&\lesssim \mathcal E_{\mu_{J,h,M}}(\sqrt f)
+\frac Rn\rho^{R-r_0}
\E_{\mu_{J,h,M}}[(\sqrt f-a)^2].
\end{aligned}
\]
Since $R-r_0\geq R/3$ and $\rho$ is universal,
$R\rho^{R-r_0}\lesssim1$. Moreover, since
$a^2=\E_{\mu_{J,h,M}}[f]$, Jensen's inequality gives
$\E_{\mu_{J,h,M}}[\sqrt f]\leq a$, and hence
\[
\begin{aligned}
\E_{\mu_{J,h,M}}[(\sqrt f-a)^2]
\leq2\left(a^2-
\bigl(\E_{\mu_{J,h,M}}[\sqrt f]\bigr)^2\right)=2\Var_{\mu_{J,h,M}}(\sqrt f)
\leq2\Ent_{\mu_{J,h,M}}(f).
\end{aligned}
\]
Combining these estimates, we obtain
\[
\mathcal Q_M(\sqrt f;a)
\lesssim \mathcal E_{\mu_{J,h,M}}(\sqrt f)
+\frac1n\Ent_{\mu_{J,h,M}}(f).
\]
Since $p\, \mathcal{E}_{\mu_{J,h,M}}(\sqrt{f}) \leq \mathcal{E}_{\mu_{J,h,S}}(\sqrt{f})$, it follows that
\[
   \E_{\mu_{J,h,S}}\left[
        q_-(\sqrt f-a)^2\one_{\Omega_M}
    \right] =  p \mathcal{Q}_M(\sqrt{f}; a)
    \lesssim
        \mathcal E_{\mu_{J,h,S}}(\sqrt f)
        +\frac{p}n
            \Ent_{\mu_{J,h,M}}(f).
\]
Substituting this into \eqref{ineq:entropy-drift-lemma-final} and using $R=(n-M)/2\ge\vartheta n/2$, hence $n/R\le2/\vartheta$, we obtain
\[
        n\,\E_{\mu_{J,h,S}}\left[
        (\sqrt f-a)^2
        \one_{\widehat\Omega_S\setminus\widehat\Omega_M}
    \right]
     \lesssim_{\zeta,\vartheta}
        n\mathcal E_{\mu_{J,h,S}}(\sqrt f) +p\Ent_{\mu_{J,h,M}}(f),
\]
which completes the proof.
\end{proof}

\begin{proof}[Proof of Theorem~\ref{theorem:lifting}-(b)]
Let $f:\widehat\Omega_S\to[0,\infty)$ and $a := (\E_{\mu_{J,h,M}} f)^{1/2}$. We first derive the decomposition in (\ref{entropy-lifting-decomposition}), for which we will use entropy chain rule. Let $\tilde{f} = a^2 \, \one_{\capOmega_M} + f \, \one_{\capOmega_S \setminus \capOmega_M}$. Using entropy chain rule, 
\begin{equation}\label{eq:entropy-chain-rule}
    \Ent_{\mu_{J,h,S}}(f) = p \, \Ent_{\mu_{J,h,M}}(f) + \Ent_{\mu_{J,h,S}}(\tilde{f})
\end{equation}
We bound the second term. By Rothaus' lemma (see e.g.~\cite[Lemma 5.1.4]{bakry2014analysis}),
\begin{equation}\label{ineq:rothhaus-entropy}
    \Ent_{\mu_{J,h,S}}(\tilde{f}) \leq \Ent_{\mu_{J,h,S}}[(\sqrt{f} - a)^2 \; \one_{\capOmega_S \setminus \capOmega_M}] + 2 \, \E_{\mu_{J,h,S}}[(\sqrt{f} - a)^2 \, \one_{\capOmega_S \setminus \capOmega_M}]
\end{equation}
If $\E_{\mu_{J,h,S}}[(\sqrt{f} - a)^2 \, \one_{\capOmega_S \setminus \capOmega_M}] = 0$, then $\widetilde{f } = a^2$, so $\Ent_{\mu_{J,h,S}}(\tilde{f}) = 0$. Otherwise, note that for all $\sigma \in \capOmega_S \setminus \capOmega_M$,
\[
    \frac{(\sqrt{f(\sigma)} -a)^2}{\E_{\mu_{J,h,S}}[(\sqrt{f} - a)^2 \, \one_{\capOmega_S \setminus \capOmega_M}]} \leq \frac{1}{\mu_{J,h,S}(\sigma)} \leq e^{\eta n}
\]
where the second inequality comes from \eqref{assumption:H4}. Thus, 
\begin{align*}
        \Ent_{\mu_{J,h,S}}[(\sqrt{f} - a)^2 \; \one_{\capOmega_S \setminus \capOmega_M}] &\leq \sum_{\sigma \in \capOmega_S \setminus \capOmega_M} \mu_{J,h,S}(\sigma)[\sqrt{f(\sigma)} - a]^2 \, \log \frac{1}{\mu_{J,h,S}(\sigma)} \\ & \leq \eta n \, \E_{\mu_{J,h,S}}[(\sqrt{f}-a)^2 \, \one_{\capOmega_S \setminus \capOmega_M}]
\end{align*}
Therefore,  the inequality in (\ref{ineq:rothhaus-entropy}) becomes
\[
    \Ent_{\mu_{J,h,S}}(\tilde{f}) \leq (\eta n+2) \, \E_{\mu_{J,h,S}}[(\sqrt{f} -a)^2 \, \one_{\capOmega_S \setminus \capOmega_M}]
\]
Using this bound in \eqref{eq:entropy-chain-rule} proves
\eqref{entropy-lifting-decomposition}. Increasing $C_1$ if necessary, set
\[
    C_{\mathrm E}:=C_1n\exp\!\left(
        \frac{C_2\|J\|_{\op}}{\eps}\right)\geq n.
\]
From the proof of Theorem~\ref{theorem:core-func-ineq}-(b), we obtain
$\Ent_{\mu_{J,h,M}}(f)\leq
 C_{\mathrm E}\mathcal H_{\mu_{J,h,M}}(f)$.
Keeping only the conditional entropy contributions from pairs contained in $\widehat\Omega_M$, and
$\Var(\sqrt{f})\leq\Ent(f)$ give
\[
    p\mathcal H_{\mu_{J,h,M}}(f)
       \leq\mathcal H_{\mu_{J,h,S}}(f),
    \qquad
    \mathcal E_{\mu_{J,h,S}}(\sqrt{f})
       \leq\mathcal H_{\mu_{J,h,S}}(f).
\]
Since \eqref{assumption:stability-entropy} implies \eqref{assumption:stability},
Lemma~\ref{lemma:basin-entropy} now yields
\begin{align*}
    \Ent_{\mu_{J,h,S}}(f)
    \leq C_{\zeta,\vartheta,\eta}\left(
        n\mathcal E_{\mu_{J,h,S}}(\sqrt{f})
        +p\Ent_{\mu_{J,h,M}}(f)\right)
    \leq C_{\zeta,\vartheta,\eta}(n+C_{\mathrm E})
        \mathcal H_{\mu_{J,h,S}}(f).
\end{align*}
Since
$\mathcal H_{\mu_{J,h,S}}(f)
\leq\mathcal E_{\mu_{J,h,S}}(f,\log f)$,
this proves the claimed MLSI bound.
\end{proof}

\section{Application estimates and proofs}\label{sec:applications}

In this section, we prove Proposition~\ref{prop:z2-inputs}, Theorem~\ref{theorem:rrg-ordering}, and Corollaries~\ref{cor:high-field-func-ineq}, \ref{corollary:SK-high-field-mlsi}. For Gaussian applications, let
$Z\sim\mathcal N(0,1)$ and let $\Phi$ denote its distribution
function. For $W\sim\mathrm{GOE}(n)$, write
\[
    \widetilde{W}:=W-\diag(W),
    \qquad
    g_i:=\sum_{j\neq i}W_{ij}=(\widetilde{W}\1)_i.
\]

\subsection{$\Z_2$ synchronization}\label{subsec:z2-estimates}

\begin{proof}[Proof of Proposition~\ref{prop:z2-inputs}]
Recall from Section~\ref{subsec:z2-proof}, we defined $J_0$ as
\[
    J_0=\frac{\lambda^2}{2n}\1\1^\top
        +\frac{\lambda}{2}\operatorname{diag}(\theta)
        W\operatorname{diag}(\theta).
\]
Conditional on $\theta$, the matrix $\diag(\theta)W\diag(\theta)$ has the GOE law,
so it suffices to take $\theta=\1$. Let $m_0,R,M$ be as in the statement. Choose $C'>0$ and then
$\lambda_0\geq 2C'$ sufficiently large as below, and fix
$\lambda\geq\lambda_0$ throughout. The standard GOE norm bound (see e.g. \cite[Proposition~2.2]{HS25})
gives a universal $K>0$ such that, for every fixed $\rho\in(0,1/2)$, with probability $1-o(1)$,
\begin{equation}\label{eq:z2-gaussian-width}
    \max\{\|W\|_{\op},\|\widetilde W\|_{\op}\}\leq K,
    \qquad
    \max_{\substack{T\subseteq[n]\\ |T|\leq\rho n}}
\max\{\|W_{T,T}\|_{\mathrm{op}},
       \|\widetilde W_{T,T}\|_{\mathrm{op}}\}
\leq K\sqrt{\rho\log\frac{e}{\rho}}.
\end{equation}
Using $\rho=2/\lambda^3$ and  $2R\leq\rho n$,
we obtain
\[
    8\omega_M(J_0)
    \lesssim\frac{\lambda^2R}{n}
        +\lambda\sqrt{\rho\log\frac e\rho}
    \lesssim\frac1\lambda
        +\sqrt{\frac{\log\lambda}{\lambda}}
    \leq\frac12.
\]
The row sums have the representation
\[
    (g_i)_{i=1}^n
    \stackrel{d}{=}
    \left(\sqrt{\frac{n-2}{n}}\,Z_i+\frac{Z_0}{\sqrt n}
    \right)_{i=1}^n,
\]
where $Z_i \iid \mathcal{N}(0,1)$.The law of large numbers and continuity of $\Phi$ give,
for every fixed $u\in\R$,
\begin{equation}\label{eq:z2-gaussian-empirical}
    \frac1n\bigl|\{i:g_i<u\}\bigr|
    \xrightarrow{\P}\Phi(u).
\end{equation}
Similarly, since $|\psi'|\leq3$, the same representation
and the law of large numbers give, for every fixed
$h_0,\beta\geq0$,
\begin{equation}\label{eq:z2-gaussian-entropy-general}
    \frac1n\sum_{i=1}^n\psi(h_0+\beta g_i)
    \xrightarrow{\P}\EE[\psi(h_0+\beta Z)].
\end{equation}
At the all-plus configuration,
$\ell_i(\1)=\lambda^2(1-1/n)+\lambda g_i$.
Since
$\psi(u)\leq(2/e)e^{-u}$,
\eqref{eq:z2-gaussian-entropy-general} and $|\psi'|\leq3$
give
\[
    \frac1n\sum_{i=1}^n\psi(\ell_i(\1))
    =\EE[\psi(\lambda^2+\lambda Z)]+o_{\P}(1)
    \leq\frac2e e^{-\lambda^2/2}+o_{\P}(1).
\]
Increasing $\lambda_0$, we therefore have, with probability
$1-o(1)$,
\[
    \sum_{i=1}^n\psi(\ell_i(\1))
    \leq\frac{\delta_*n}{2\lambda^3}
    \leq\delta_*R,
\]
where $R\geq n/(2\lambda^3)$ for all sufficiently large $n$.
This verifies \eqref{assumption:stability-entropy}.

To prove \eqref{eq:z2-input-drift}, fix $\sigma$ with
$m_0\leq m(\sigma)<M/n$, and put $u:=m(\sigma)$,
$T:=\{i:\sigma_i=-1\}$, and $r:=|T|>R$.
Since $\sigma=\1-2\one_T$,
\[
    \ell_i(\sigma)
    =\lambda^2\left(u-\frac{\sigma_i}{n}\right)
        +\lambda g_i-2\lambda(\widetilde W\one_T)_i.
\]
Applying \eqref{eq:z2-gaussian-empirical} at
$-C'/8$ and $-\lambda/8$, we obtain, with probability
$1-o(1)$, uniformly over these configurations,
\[
    \left|\left\{i:g_i<-\frac{\lambda u}{4}\right\}\right|
    \leq
    \begin{cases}
        2n\Phi(-C'/8)\leq r/8, &u\leq1/2,\\[2mm]
        2n\Phi(-\lambda/8)\leq r/8, &u>1/2.
    \end{cases}
\]
Indeed, $r\geq n/4$ in the first case, whereas
$r>R\geq n/(2\lambda^3)$ in the second.
The bounds follow by taking $C'$ and then $\lambda_0$
sufficiently large, since $\Phi(-\lambda/8)=o(\lambda^{-3})$.
Also, \eqref{eq:z2-gaussian-width} gives
\[
    \left|\left\{i:
        |(\widetilde W\one_T)_i|>\frac{\lambda u}{8}
    \right\}\right|
    \leq\frac{64\|\widetilde W\one_T\|_2^2}{\lambda^2u^2}
    \leq\frac{256K^2}{(C')^2}r
    \leq\frac r8
\]
for sufficiently large $C'$.
Thus, at all but at most $r/4$ coordinates,
\[
    \ell_i(\sigma)
    \geq\lambda^2\left(u-\frac1n\right)
        -\frac{\lambda^2u}{2}
    \geq\frac{\lambda^2u}{4}
\]
for all sufficiently large $n$. Consequently,
\[
    \sum_{i=1}^n\phi(\ell_i(\sigma))
    \leq\frac r4+ne^{-\lambda^2u/2}
    \leq\frac r2.
\]
For the last inequality, use $u\geq C'/(2\lambda)$
and $r\geq n/4$ when $u\leq1/2$, and
$e^{-\lambda^2/4}=o(\lambda^{-3})$ and
$r>R\geq n/(2\lambda^3)$ otherwise, increasing
$\lambda_0$ if necessary.
Therefore,
\[
    n\bigl((Pm)(\sigma)-m(\sigma)\bigr)
    =\frac2n\left(r-\sum_{i=1}^n\phi(\ell_i(\sigma))\right)
    \geq\frac rn
    =\frac{1-m(\sigma)}2,
\]
uniformly over the stated configurations.
\end{proof}

\subsection{Ferromagnetic Ising model on a random regular graph}
We prove Theorem~\ref{theorem:rrg-ordering}. For a graph $G$ on
$[n]$, let $A_G$ denote its adjacency matrix, with
$(A_G)_{ij}=\one_{\{i\sim j\}}$. The following lemma gives the
spectral and edge expansion estimates used below.

\begin{lemma}\label{lemma:rrg-spectral-geometry}
There is a universal constant $K>0$ such that, for every fixed
integer $d\geq3$ and $G\sim\mathcal G_d(n)$, the following hold
simultaneously with probability $1-o(1)$ as $n\to\infty$ through
integers for which $nd$ is even. For every $1\leq s\leq n$,
\begin{equation}\label{ineq:rrg-sparse-width}
    \omega_{n-s}(A_G)\leq\frac{ds}{n}+2K\sqrt d.
\end{equation}
Moreover, every $S\subseteq[n]$, with $k:=|S|$, satisfies
\begin{equation}\label{ineq:rrg-expansion}
    |\partial S|
    \geq(d-K\sqrt d)k\left(1-\frac kn\right),
\end{equation}
where $\partial S$ is the set of edges joining $S$ to
$[n]\setminus S$.
\end{lemma}
\begin{proof}
By Friedman's estimate~\cite[Corollary~1.4]{Friedman}, with
probability $1-o(1)$,
\begin{equation}\label{eq:rrg-spectral-bound}
    \left\|A_G-\frac dn\1\1^\top\right\|_{\op}
    \leq\sqrt d+\sqrt{d-1}<2\sqrt d.
\end{equation}
We show that
\eqref{eq:rrg-spectral-bound} implies both conclusions with $K=2$. For a nonempty $T\subseteq[n]$ and $x\in\R^n$ supported on $T$,
\[
    x^\top A_Gx
    =\frac dn\langle x,\1\rangle^2
       +x^\top\left(A_G-\frac dn\1\1^\top\right)x,
    \qquad
    \langle x,\1\rangle^2\leq|T|\|x\|_2^2.
\]
Consequently,
\[
    \lambda_{\min}((A_G)_{T,T})
    \geq-\sqrt d-\sqrt{d-1},
    \qquad
    \lambda_{\max}((A_G)_{T,T})
    \leq\frac{d|T|}{n}+\sqrt d+\sqrt{d-1}.
\]
Subtracting and maximizing over $1\leq|T|\leq s$ gives
\[
    \omega_{n-s}(A_G)
    \leq\frac{ds}{n}+2\bigl(\sqrt d+\sqrt{d-1}\bigr)
    \leq\frac{ds}{n}+2K\sqrt d,
\]
which proves \eqref{ineq:rrg-sparse-width}. For \eqref{ineq:rrg-expansion}, fix $S\subseteq[n]$, put $k:=|S|$,
and set $x:=\one_S-(k/n)\1$. Then
\[
    x\perp\1,
    \qquad
    \|x\|_2^2=k\left(1-\frac kn\right),
    \qquad
    (dI_n-A_G)\1=0.
\]
Thus, by \eqref{eq:rrg-spectral-bound},
\[
    \begin{aligned}
        |\partial S|
        =\one_S^\top(dI_n-A_G)\one_S
          =x^\top(dI_n-A_G)x
        \geq(d-K\sqrt d)\|x\|_2^2
          =(d-K\sqrt d)k\left(1-\frac kn\right),
    \end{aligned}
\]
which proves \eqref{ineq:rrg-expansion}.
\end{proof}

\begin{lemma}\label{lemma:local-field-magnetization-drift}
Let $d,\beta$ be as in Theorem~\ref{theorem:rrg-ordering},
with $c_0\geq6$, and let $G\sim\mathcal G_d(n)$.
Let $P$ denote the Glauber kernel for $\pi_{G,\beta}$.
There exists $\delta=\delta(\beta,d)>0$ such that, with
probability $1-o(1)$ as $n\to\infty$ through integers for which
$nd$ is even,
\begin{equation}\label{ineq:local-field-magnetization-drift}
    n\bigl((Pm)(\sigma)-m(\sigma)\bigr)\geq\delta,
    \qquad
    \frac{3\beta}{2}\leq m(\sigma)\leq1-2e^{-\beta d}.
\end{equation}
\end{lemma}

\begin{proof}
Put $\kappa:=\beta d$. The assumptions give
$\kappa\geq6$ and $\beta^2d\leq1/256$.
By \eqref{eq:rrg-spectral-bound}, with probability $1-o(1)$,
\[
    \left\|A_G-\frac dn\1\1^\top\right\|_{\op}<2\sqrt d.
\]
Fix $\sigma\in\Omega$ in the stated range, and set $u:=m(\sigma)$, and $H_i:=\beta\sum_{j\sim i}\sigma_j$. Then
\[
    n\bigl((Pm)(\sigma)-m(\sigma)\bigr)
    =\frac1n\sum_{i=1}^n\tanh H_i-u.
\]
Since $A_G\1=d\1$, $\sigma-u\1\perp\1$, and
$\|\sigma-u\1\|_2^2=n(1-u^2)$,
\[
    \begin{aligned}
        \frac1n\sum_{i=1}^nH_i
        =\frac{\beta}{n}\1^\top A_G\sigma=\kappa u, \quad \text{and} \quad 
        \frac1n\sum_{i=1}^n(H_i-\kappa u)^2
        =\frac{\beta^2}{n}\|A_G(\sigma-u\1)\|_2^2
          \leq4\beta^2d(1-u^2).
    \end{aligned}
\]
Suppose first that $\kappa u\leq1$.
The elementary inequality $\tanh x\geq x-3x^2/10$,
$x\in\R$, gives
\[
    \begin{aligned}
        n\bigl((Pm)(\sigma)-m(\sigma)\bigr)
        \geq(\kappa-1)u-\frac3{10n}\sum_{i=1}^nH_i^2
        \geq(\kappa-1)u
          -\frac3{10}\bigl(\kappa^2u^2+4\beta^2d\bigr).
    \end{aligned}
\]
The last expression is concave in $u$. Evaluating it at
$u=3\beta/2$ and $u=1/\kappa$, using $\kappa\geq6$ and
$\beta^2d\leq1/256$, gives
\[
    n\bigl((Pm)(\sigma)-m(\sigma)\bigr)
    \geq\frac{\beta^2d}{32},
    \qquad \frac{3\beta}{2}\leq u\leq\frac1\kappa.
\]
Now suppose that $\kappa u>1$.
Since $\|\tanh''\|_\infty<1$ and
$\sum_i(H_i-\kappa u)=0$, Taylor's theorem gives
\[
    \begin{aligned}
        \frac1n\sum_{i=1}^n\tanh H_i
        \geq\tanh(\kappa u)
          -\frac1{2n}\sum_{i=1}^n(H_i-\kappa u)^2
        \geq\tanh(\kappa u)-2\beta^2d(1-u^2).
    \end{aligned}
\]
For $\kappa\geq6$, the function
$u\mapsto\tanh(\kappa u)-(1+u)/2$ is concave on
$[1/\kappa,1-2e^{-\kappa}]$ and positive at both endpoints.
Hence
\[
    \tanh(\kappa u)\geq\frac{1+u}{2},
    \qquad \frac1\kappa\leq u\leq1-2e^{-\kappa}.
\]
Consequently,
\[
    \begin{aligned}
        n\bigl((Pm)(\sigma)-m(\sigma)\bigr)
        \geq\frac{1-u}{2}-2\beta^2d(1-u^2)
        \geq\left(\frac12-4\beta^2d\right)(1-u)
          \geq\frac{1-u}{4}
          \geq\frac12e^{-\kappa}.
    \end{aligned}
\]
Combining the two cases proves the lemma with $ \delta:=\frac1{32}\min\{\beta^2d,e^{-\kappa}\}>0$.
\end{proof}
\begin{proof}[Proof of Theorem~\ref{theorem:rrg-ordering}]
Choose $c_0,d_0$ sufficiently large. We use the estimates of
Lemmas~\ref{lemma:rrg-spectral-geometry}
and~\ref{lemma:local-field-magnetization-drift}, which hold
simultaneously with probability $1-o(1)$. Put
\[
    \kappa:=\beta d,
    \qquad J:=\frac{\beta}{2}A_G,
    \qquad R:=\lfloor ne^{-\kappa}\rfloor,
    \qquad L:=n-2R.
\]
For the cutoff $M$ in the statement, let
\[
    \capOmega_M
      =\{\sigma\in\Omega:m(\sigma)\geq3\beta/2\},
    \qquad
    \mu:=\pi_{G,\beta}(\,\cdot\mid \capOmega_M),
    \qquad
    \widehat P:=P_\mu.
\]
Let $P$ denote the Glauber kernel for $\pi_{G,\beta}$.
Here and below, $C$ depends only on $\beta,d$.
For all sufficiently large $n$,
\[
    0\leq M<L<n,
    \qquad
    R\geq\frac n2e^{-\kappa},
    \qquad
    L\leq(1-e^{-\kappa})n.
\]
We verify the hypotheses of Theorem~\ref{theorem:lifting}-(b)
for $\capOmega_L\subset \capOmega_M$. The estimate in the proof of
Lemma~\ref{lemma:rrg-spectral-geometry} gives
\[
    \begin{aligned}
        8\omega_L(J)
        \leq\frac{8\beta dR}{n}
          +8\beta\bigl(\sqrt d+\sqrt{d-1}\bigr)
        \leq8\kappa e^{-\kappa}
          +\frac{16\kappa}{\sqrt d}
             \left(1-\frac1{4d}\right)
        \leq1-\frac1{8d}.
    \end{aligned}
\]
The last bound follows by convexity in $\kappa$ and evaluation
at the endpoints of $[c_0,\sqrt d/16]$, for sufficiently large
$c_0,d_0$. At $\1$, we have $\ell_i(\1)=\kappa$, and hence
\[
    \sum_{i=1}^n\psi(\ell_i(\1))
    =n\psi(\kappa)
    \leq(2\kappa+1)ne^{-2\kappa}
    \leq\delta_*R,
\]
for sufficiently large $c_0$ and $n$. Thus
\eqref{assumption:stability-entropy} holds at cutoff $L$. Let $\delta\in(0,1)$ be as in
Lemma~\ref{lemma:local-field-magnetization-drift}.
For $\sigma\in\Omega_k$ with $M\leq k<L$, 
\[
    \frac{3\beta}{2}\leq m(\sigma)=\frac kn
    \leq1-\frac{2(R+1)}n<1-2e^{-\kappa}.
\]
Consequently,
\[
    \begin{aligned}
      q_{+}(\sigma)
        -q_{-}(\sigma)\one_{\{k>M\}}
        \geq\frac n2\bigl((Pm)(\sigma)-m(\sigma)\bigr)
        \geq\frac{\delta}{2}
        \geq\frac{\delta(n-k)}{2n},
    \end{aligned}
\]
which establishes \eqref{assumption:tail-stability}.
Moreover, $\|J\|_{\op}=\kappa/2$ and $|\capOmega_M|\leq2^n$ give
\begin{equation}\label{eq:rrg-atom-bound}
    \max_{\sigma\in \capOmega_M}\log\frac1{\mu(\sigma)}
    \leq n\log2+2n\|J\|_{\op}
    =(\kappa+\log2)n.
\end{equation}
This verifies \eqref{assumption:H4}. Therefore, Theorem~\ref{theorem:lifting}-(b) proves
\eqref{eq:rrg-basin-mlsi}. In particular, by the resulting MLSI bound, \eqref{eq:rrg-atom-bound}, and
\eqref{eq:mixing-time-bounds}, after increasing $C$, we have
\begin{equation}\label{eq:rrg-reflected-mixing}
    \sup_{\sigma\in \capOmega_M}
    \|\widehat P^t(\sigma,\cdot)-\mu\|_{\TV}
    \leq n^{-12},\qquad \forall t\geq Cn\log n.
\end{equation}
For $\sigma\in\Omega^+\setminus \capOmega_M$, put
$T:=\{i:\sigma_i=-1\}$. Since
$0\leq m(\sigma)<3\beta/2\leq1/3$, we have
$n/3\leq|T|\leq n/2$. For sufficiently large $d_0$,
Lemma~\ref{lemma:rrg-spectral-geometry} gives $ |\partial T| \geq \frac{dn}{6}$. Therefore,
\[
    \frac{\pi_{G,\beta}(\sigma)}{\pi_{G,\beta}(\1)}
    =e^{-2\beta|\partial T|}
    \leq e^{-\kappa n/3}\leq e^{-2n},
\]
and
\begin{equation}\label{eq:rrg-phase-mass}
    \begin{aligned}
        \|\mu-\pi_{G,\beta}^+\|_{\TV}
        =\pi_{G,\beta}^+(\capOmega_M^c)
        \leq\sum_{\sigma\in\Omega^+\setminus \capOmega_M}
          \frac{\pi_{G,\beta}(\sigma)}{\pi_{G,\beta}(\1)}
        \leq2^ne^{-2n}\leq e^{-n}.
    \end{aligned}
\end{equation}
Finally, define
\[
    \tau_0:=\inf\{t\geq0:X_t\notin \capOmega_M\}
           =\inf\{t\geq0:m(X_t)<3\beta/2\}.
\]
Since $5\beta/2<1-2e^{-\kappa}$,
Lemma~\ref{lemma:local-field-magnetization-drift} applies
throughout $[3\beta/2,5\beta/2]$. This interval has midpoint
$2\beta$, so Lemma~\ref{lemma:basin-retention} gives, for some
$\gamma=\gamma(\beta,d)>0$,
\begin{equation}\label{eq:rrg-basin-retention}
    \sup_{\sigma:\,m(\sigma)\geq2\beta}
    \P_\sigma(\tau_0\leq e^{\gamma n})
    \leq e^{-\gamma n}.
\end{equation}
Couple the chains with kernels $P$ and $\widehat P$ using
the same updates. They agree until $\tau_0$, and hence
\[
    \|P^t(\sigma,\cdot)-\widehat P^t(\sigma,\cdot)\|_{\TV}
    \leq\P_\sigma(\tau_0\leq t).
\]
Increase $C$ so that $C\geq\gamma^{-1}$.
For $m(\sigma)\geq2\beta$ and integer
$Cn\log n\leq t\leq e^{n/C}$,
\[
    \begin{aligned}
        \|P^t(\sigma,\cdot)-\pi_{G,\beta}^+\|_{\TV}
        &\leq\P_\sigma(\tau_0\leq t)
          +\|\widehat P^t(\sigma,\cdot)-\mu\|_{\TV}
          +\|\mu-\pi_{G,\beta}^+\|_{\TV}\\
        &\leq e^{-\gamma n}+n^{-12}+e^{-n}
        \leq n^{-10},
    \end{aligned}
\]
for all sufficiently large $n$, proving
\eqref{eq:rrg-phase-ordering}.
\end{proof}

\subsection{High-field functional inequalities}

We first prove Corollary~\ref{cor:high-field-func-ineq} by verifying
the assumptions of Theorem~\ref{theorem:lifting}, and then apply it
to prove Corollary~\ref{corollary:SK-high-field-mlsi}.

\begin{proof}[Proof of Corollary~\ref{cor:high-field-func-ineq}]
Choose $0<c\leq\delta_*/4$, and recall that $\gamma=\omega_0(J)$.
Put
\[
    J':=J-\frac{\lambda_{\max}(J)+\lambda_{\min}(J)}2I_n,
    \qquad A:=J-\diag(J).
\]
The scalar shift preserves the measure, local fields, and sparse
spectral widths. Moreover, $\|J'\|_{\op}=\gamma/2$ and
$\|A\|_{\op}\leq2\|J'\|_{\op}=\gamma$.
Set
\[
    R:=\left\lfloor\frac{\delta n}{2}\right\rfloor,
    \qquad M:=n-2R.
\]
Since $n\geq2/\delta$, we have $R\geq1$,
$R\geq\delta n/4$, and $2R\leq\lfloor\delta n\rfloor$. Thus,
\[
    8\omega_M(J')
    =8\omega_M(J)
    \leq8\omega_{n-\lfloor\delta n\rfloor}(J)
    \leq1-\eps.
\]
By monotonicity of $\phi$ and
\eqref{assumption:high-field-corollary},
\[
    \sum_{i=1}^n\phi(\ell_i(\1))
    \leq\sum_{i=1}^n\phi(\ell_i(\1)-8\gamma)
    \leq c\delta n\leq\delta_*R.
\]
Hence \eqref{assumption:stability} holds at cutoff $M$. To verify \eqref{assumption:tail-stability}, fix
$\sigma\notin\capOmega_M$, put $T:=\{i:\sigma_i=-1\}$ and
$r:=|T|>R$, and let $E:=\{i:4(A\one_T)_i>8\gamma\}$.
Since $\ell_i(\sigma)=\ell_i(\1)-4(A\one_T)_i$ and
$\|A\one_T\|_2^2\leq\gamma^2r$, for $\gamma>0$,
\[
    |E|\leq\frac{\|A\one_T\|_2^2}{4\gamma^2}\leq\frac r4.
\]
If $\gamma=0$, then $A=0$ and $E=\varnothing$.
For $i\notin E$,
$\phi(\ell_i(\sigma))
 \leq\phi(\ell_i(\1)-8\gamma)$.
Consequently,
\[
    \sum_{i=1}^n\phi(\ell_i(\sigma))
    \leq\frac r4+
        \sum_{i=1}^n\phi(\ell_i(\1)-8\gamma)
    \leq\frac r4+c\delta n
    \leq\frac r4+\delta_*R
    \leq\frac r2.
\]
Let $P$ denote the Glauber kernel for $\mu_{J,h}$. Then, for $\sigma\in\Omega_k$ with
$-n\leq k<M$,
\[
    \begin{aligned}
        q_{+}(\sigma)-q_{-}(\sigma)
        =\frac1n\left(
            r-\sum_{i=1}^n\phi(\ell_i(\sigma))\right)
        \geq\frac r{2n}=\frac{n-k}{4n}.
    \end{aligned}
\]
Thus \eqref{assumption:tail-stability} holds with outer cutoff
$S=-n$. Since $\capOmega_S=\Omega$ and
$\mu_{J',h}=\mu_{J,h}$, Theorem~\ref{theorem:lifting}-(a)
gives $\gap(\mu_{J,h})\gtrsim_{\eps,\gamma}1/n$.

For the entropy estimate, the additional assumption gives
$\sum_i\psi(\ell_i(\1))\leq c\delta n\leq\delta_*R$,
which is \eqref{assumption:stability-entropy}.
Using $\psi(u)\geq(-u)_+$ and
$\ell_i(\1)=h_i+2(A\1)_i$, we obtain
\[
    \begin{aligned}
        \sum_{i=1}^n(-h_i)_+
        \leq\sum_{i=1}^n(-\ell_i(\1))_+
             +2\|A\1\|_1
        \leq\sum_{i=1}^n\psi(\ell_i(\1))
             +2\sqrt n\,\|A\1\|_2
        \leq c\delta n+2\gamma n.
    \end{aligned}
\]
Since $\sum_i|h_i|=\sum_i h_i+2\sum_i(-h_i)_+$ and
$\mu_{J,h}=\mu_{J',h}$,
\[
    \begin{aligned}
        \max_{\sigma\in\Omega}\log\frac1{\mu_{J,h}(\sigma)}
        \leq n\log2+2n\|J'\|_{\op}
             +2\sum_{i=1}^n|h_i|
        &\leq C(1+\gamma+H)n.
    \end{aligned}
\]
This verifies \eqref{assumption:H4}. Since
$M=n-2R\leq(1-\delta/2)n$, Theorem~\ref{theorem:lifting}-(b)
gives $\rho_{\mathrm{LS}}(\mu_{J,h})
\gtrsim_{\eps,\gamma,\delta,H}1/n$, completing the proof.  
\end{proof}

\begin{proof}[Proof of Corollary~\ref{corollary:SK-high-field-mlsi}]
Set $J:=\beta\widetilde W/2$, so that
$\mu_{J,h}=\mu_{\beta W/2,h}$.
Let $K$ be the constant in \eqref{eq:z2-gaussian-width}
and let $c>0$ be the constant in
Corollary~\ref{cor:high-field-func-ineq}.
Choose a universal constant $C>0$ sufficiently large as below,
and put
\[
    \delta=\frac{1}{C(1+\beta)^4},
    \qquad
    h_0=C\left(1+\beta\sqrt{\log(1+\beta)}\right).
\]
Then $8K\beta\sqrt{\delta\log(e/\delta)}\leq1/2$. Since $0\leq-\psi'\leq3$ and $\psi(u)\leq(2/e)e^{-u}$,
for $u,\beta\geq0$,
\[
    \begin{aligned}
        \EE[\psi(u+\beta Z)]
        \leq \psi(u/2)+3\EE[(-u/2-\beta Z)_+]
        \lesssim e^{-u/2}
            +\beta\exp\!\left(-\frac{u^2}{8(1+\beta^2)}\right).
    \end{aligned}
\]
Taking $C$ sufficiently large ensures
$h_0-8K\beta\geq h_0/2$ and $\EE[\psi(h_0-8K\beta+\beta Z)] \leq c \delta / 2$. By \eqref{eq:z2-gaussian-width} and
\eqref{eq:z2-gaussian-entropy-general}, with probability
$1-o(1)$,
\[
    \omega_0(J)\leq2\|J\|_{\op}\leq K\beta,
    \qquad
    8\omega_{n-\lfloor\delta n\rfloor}(J)\leq\frac12,\qquad \sum_{i=1}^n\psi(h_0-8K\beta+\beta g_i)\leq c\delta n.
\]
For every $h\in\R^n$ with $\min_i h_i\geq h_0$ and
$n^{-1}\sum_i h_i\leq H$, monotonicity gives
\[
    \sum_{i=1}^n\psi(h_i+\beta g_i-8\omega_0(J))
    \leq\sum_{i=1}^n\psi(h_0-8K\beta+\beta g_i)
    \leq c\delta n.
\]
Since $\phi\leq\psi$ and $\psi$ is decreasing, both field
conditions in Corollary~\ref{cor:high-field-func-ineq} hold.
Thus $\rho_{\mathrm{LS}}(\mu_{J,h})
\gtrsim_{\beta,H}n^{-1}$.
Since $h_i\geq0$,
\[
    \max_{\sigma\in\Omega}\log\frac1{\mu_{J,h}(\sigma)}
    \leq n\log2+2n\|J\|_{\op}+2\sum_{i=1}^n h_i
    \leq n(\log2+K\beta+2H).
\]
The MLSI bound and \eqref{eq:mixing-time-bounds} therefore give
\[
    t_{\mathrm{mix}}(P_{\mu_{J,h}},\eta)
    \leq C_{\beta,H}n
        \left(\log n+\log\frac1\eta\right),
    \qquad 0<\eta\leq\frac12.
\]
The event above depends only on $W$ and $\beta$, so all estimates
hold uniformly over the stated external fields.
\end{proof}

\subsection*{Statement of AI use}
The proofs in this work were developed through discussions with ChatGPT 5.5 Pro, motivated by the phase ordering of $\mathbb{Z}_2$ synchronization. The authors proposed conditioning on a core $\widehat\Omega$, establishing spectral independence within that core, introducing the quadratic boundary form $\mathcal Q_{\nu}^{\widehat\Omega}(f;a)$ to account for transitions outside the core, and applying linear-tilt stochastic localization scheme. Subsequent discussions with ChatGPT 5.5 Pro led to the boundary absorption argument, the extension to MLSI, and the use of Lyapunov conditions to lift the core functional inequalities to the hypercube. The authors wrote the initial manuscript, and then revised it and checked the math with assistance from ChatGPT 5.6 and ChatGPT 6. All AI-assisted calculations and suggestions were carefully checked by the authors, who take full responsibility for the contents of the paper.

\subsection*{Acknowledgements} We thank Reza Gheissari for introducing the high-field SK mixing problem during his visit to Brown, and Dan Mikulincer for helpful discussions on stochastic localization. Youngtak Sohn is supported by NSF grant DMS-2601870. Vasyl Paliy was partially supported by U.S. Department of Education Graduate Assistance in Areas of National Need (GAANN) award P200A240015-25.

\bibliography{main}
\bibliographystyle{alpha}
\end{document}